\documentclass[a4paper]{article}
\usepackage{amscd}
\usepackage{tikz}

\usepackage{amsfonts,amsmath,amsxtra,color,mathrsfs}
\input xy
\xyoption{all}
\usepackage[all,knot]{xy}

\newtheorem{theorem}{Theorem}[section]
\newtheorem{corollary}{Corollary}[section]
\newtheorem{lemma}{Lemma}[section]
\newtheorem{proposition}{Proposition}[section]
\newtheorem{problem}{Problem}[section]

\newenvironment{definition}
{\smallskip\noindent{\bf Definition\/}:}{\smallskip\par}
\newenvironment{proof}
{\noindent{\bf Proof\/}}{{ $\Box$}\smallskip\par}

\newenvironment{example}
{\smallskip\noindent{\bf Example\/}:}{\smallskip\par}
\newenvironment{remark}
{\smallskip\noindent{\bf Remark\/}:}{\smallskip\par}

\title{On Young diagrams, flips and cluster algebras of type A}
\author{Mikhail Gorsky}

\begin{document}
\maketitle

\begin{abstract}
We give a new simple description of the canonical bijection between the set of triangulations of $n$-gon and some set of Young diagrams. Using this description, we find flip transformations
on this set of Young diagrams which correspond to the edges of the
associahedron. This construction is generalized on the set of all Young diagrams and the corresponding infinite-dimensional associahedron is defined. We consider its relation to the properly defined
infinite-type version of the cluster algebras of type A and check some properties of these algebras inherited from their finite-type counterparts. We investigate links between these algebras and cluster categories of infinite Dynkin type $A_\infty$ introduced by Holm and Jorgensen.
\end{abstract}

\vspace{2cm}

\tableofcontents

\vspace{3cm}

\section {Introduction}
It is well-known (see, for instance, \cite{Stan}) that the Catalan numbers count both 
the triangulations of the (n+2)-gon  and the Dyck paths of the length 2n. The Dyck
paths can be considered as the paths from the point (0,-n) to the point (n,0) consisted
of the vectors (0,1) and (1.0) which never go below the line $y=x-n,$ hence they bound the Young diagrams
in the triangle formed by coordinate axes and the line $y=x-n$. Since there exists the canonical bijection \cite{Stan} between the triangulations and paths (which goes through the binary root trees with $(2n+1)$ vertices),
we can consider also the canonical bijection between the triangulations and this class of the Young diagrams (\cite{Stan}, Exercise 6.19, vv). In Section 2.1 we shall present the simple 
description of such bijection (without appealing to trees and Dyck paths) which shall be very helpful in the further analysis.

On the other hand (see \cite{Lee}) the triangulations of the (n+2)-gon enumerate the vertices of the
(n-1)-dimensional associahedron (Stasheff polytope) $As^{n-1}$ introduced in \cite{Stas}.
The edges of $As^{n-1}$ correspond to the \textit{flips} that is operation of  
the unique change of diagonal which  maps between two triangulations. Therefore
we can define the n-flips at Young diagram as follows: two Young diagrams
are related by n-flip if both of them do not intersect $y=x-n$ and 
corresponding two vertices of $As^{n-1}$ are connected by the edge. The main
result of the paper is 
 
\begin{theorem} \label{indep}
If there is the $n$-flip between two Young diagrams $A$ and $B$, there are 
also $k$-flips between them if both $A$ and $B$ are above  the line $y=x-k$,
in particular, for all $k>n$.
\end{theorem}
 
 The Theorem \ref{indep} implies that the number $n$ can be omitted and one can say
 that two Young diagrams are connected by flip if the corresponding vertices
 of the associahedron of the large dimensions are connected by the edge. In Section
 2.2 we suggest the simple explicit flips between two diagrams
 without appealing to the bijection mentioned above which can be considered
 as the operation on the set of Young diagrams. It is known that $As^{n-1}$
 is embedded as a facet into $As^{n}$ hence we can define $As^{\infty}$ as the direct limit of a filtration
 $$As^0 \hookrightarrow As^1 \hookrightarrow As^2 \hookrightarrow \ldots As^{\infty}$$ 
 and describe it in terms of Young diagrams. 
The dihedral group $D_n$ acts naturally at the associahedron $As^n$ hence we can define its action commuting with flips at the correspondent set of Young diagrams which
 the Section 2.3 is devoted to. Actually the work on this paper have been started by attempts to understand this action.
 
 In Section 3 we consider the possible connections of our construction with the cluster algebras
 which were introduced by Fomin and Zelevinsky in \cite{FZ1} and have been investigated intensively last years. They are naturally related to the quiver representations, Poisson structures,
 integrable systems, e.t.c. (see, for example, surveys \cite{FR} and \cite{Kel}). We shall provide a description of the cluster algebras
 related to $As^{\infty}$ which shall be denoted as cluster algebras of type $A_{\infty}$. 
 
 The simplest consequence of our flip construction is the nontrivial symmetry on the
 variety of the vertices of $As^n$ corresponding to the transposition of Young
 diagrams. Being extended to the $As^{\infty}$ case it yields the unexpected involution which 
 interchanges in each seed of some cluster algebras $A_{\infty}$ the frozen and 
 cluster variables. The rigorous description of this involution deserves the
 additional investigation. It will be probably more natural to consider this involution for quantum cluster algebras introduced by Berenstein and Zelevinsky in \cite{BZ}. Moreover, in 3.4 we check some properties of algebras $A_{\infty}$
 inherited from algebras of type $A_n.$ This subsection is related to work \cite{JP} (see also \cite{HJ}) where analogous results are given (inter alia) in terms of cluster categories. Our approach to algebras of type $A_{\infty}$ can be naturally generalized to types $B_{\infty}, C_{\infty}$ and $D_{\infty}$ of cluster algebras with the same properties, while question of constructing correspondent cluster categories is yet open.
 
 In our paper we have found the relation between the Stasheff polytopes and Young diagrams.
 It is known that the generating function of the number of k-dimensional faces of the
 n-dimensional associahedron obeys the Hopf equation (see \cite{B}, \cite{BK}). On the other hand the Hopf
 equation emerges while considering the representations of the algebra $SU(N)$ at $N \rightarrow \infty$  
 in terms of Young diagrams (\cite{GG}, \cite{GM}). It would be interesting to clarify the relation between
 two appearances of this universal equation using our approach.
 
 Since the cluster algebras $B_n,C_n,D_n$ correspond to some polytopes  whose 
 combinatorial structures are described in terms of triangulations (\cite{FZ3}) as well, it would be interesting
 to extend our analysis to these series.
 
 The author is grateful to his scientific supervisor V. M. Buchstaber who has attracted his attention to the associahedra
 and to M.~Z.~Shapiro, G.~I.~Olshanskii, A.~V.~Zelevinsky, E.~A.~Gorsky, A.~A.~Gaifullin and A.~V.~Fonarev for useful remarks.

\section{Flips, triangulations and Young diagrams}

\subsection{Triangulations and Young diagrams}

It will be useful to identify partition $(d_1, d_2, d_3, \ldots),$ where 
\newline $d_1 \geq d_2 \geq d_3 \geq \ldots, \quad$ with Young diagram with rows of lengths $d_1, d_2, d_3, \ldots$.

\begin{definition}
We put each Young diagram into IV quadrant and define $\mathbb{Y}_n$ as the set of all Young diagrams lying inside triangle formed by coordinate axis and line $y = x - n.$
\end{definition}

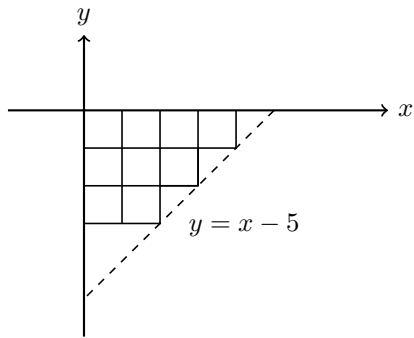
\begin{tikzpicture}[line width=0.6pt]
  \draw [<->,thick] (1,4) node (yaxis) [above] {$y$}
     |- (5,3) node (xaxis) [right] {$x$};
  \draw [thick] (1,0) -- (1,3);
  \draw [thick] (0,3) -- (1,3);
  \draw (3,3) -- (3,2.5) -- (2.5,2.5) -- (2.5,2) -- (2,2) -- (2,1.5) -- (1,1.5);
  \draw (1.5,3) -- (1.5,1.5);
  \draw (2,3) --(2,2);
  \draw (2.5,3) -- (2.5,2.5);
  \draw (1,2.5) -- (2.5,2.5);
  \draw (1,2) -- (2.5,2);
  \draw (1,2.5) -- (1,2);
  \draw (2.5,2.5) -- (2.5,2);
  \draw [dashed] (3.5,3) -- node [below right] {$y = x - 5$}(1,0.5);
\end{tikzpicture}

Consider regular $n$-gon and enumerate its vertices by nonnegative numbers $\left\{ 0, 1, \ldots, (n - 1) \right\}$ in counter-clockwise order.

\begin{definition}
The \textit{triangulation} of polygon is set of its diagonals such that 
\begin{itemize}
\item[1)] Any two of them either have common vertex or do not intersect; 
\item[2)] These diagonals divide polygon into triangles.
\end{itemize}

Denote by $T_n$ the set of all triangulations of the convex $n$-gon. 

We denote by \textit{"tail"} and \textit{"head"} of every diagonal the smallest and the biggest of its ends respectively.
\end{definition}

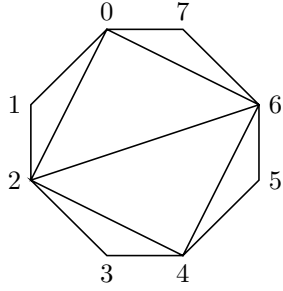
\begin{tikzpicture}[line width=0.6pt]
  \draw (0,1) node [left] {$2$} -- 
        (0,2) node [left] {$1$} -- 
        (1,3) node [above] {$0$} -- 
        (2,3) node [above] {$7$} -- 
        (3,2) node [right] {$6$} -- 
        (3,1) node [right] {$5$} -- 
        (2,0) node [below] {$4$} -- 
        (1,0) node [below] {$3$} -- 
        (0,1) -- (2,0) -- (3,2) -- (1,3) -- (0,1) -- (3,2);
\end{tikzpicture}

\begin{definition}
We correspond to each triangulation $A \in T_{n+2}$ the collection of integer numbers $\Lambda_{n+2}(A) = (\lambda_1(A), \lambda_2(A), \ldots, \lambda_{n-1}(A)),$ where each $\lambda_i(A)$ is the tail of some diagonal and $\lambda_i(A)$ are decreasingly ordered. We will omit index ${n+2}$ when its value would be clear.
\end{definition}

\begin{lemma}
For every pair of distinct triangulations $A, B \in T_n$ $\Lambda(A) \neq \Lambda(B)$.
\end{lemma}

\begin{proof}
Prove this lemma by induction. For $n = 3,4$ the statement is obvious. Suppose that it holds for $n=k.$ Suppose that for $A, B \in T_{k+1}, A \neq B$ $\quad$ $\Lambda_{k+1}(A) = \Lambda_{k+1}(B).$ Let $\lambda_1(A) = \lambda_1(B) = l.$ It is clear that in both $A$ and $B$ there is a diagonal $(l, l + 2).$ Consider triangulations $A^{'} = A \backslash \left\{(l,l+2)\right\}, B^{'} = B \backslash \left\{(l,l+2)\right\}$ of $(k-1)-$gon with vertices $\left\{0, 1, \ldots, l, l+2, l+3, \ldots, k-1\right\}.$ Since $\Lambda_{k}(A^{'}) = (\lambda_2(A), \ldots, \lambda_{k-2}(A)) = \Lambda_{k}(B^{'}),$ then by assumption $A^{'} = B^{'},$ consequently $A = B,$ q.e.d.
\end{proof}

We will enumerate diagonals of triangulation as follows: for two diagonals $(a,b)$ and $(c,d)$ (here $a < b, c < d$) $\quad (a,b)$ has bigger number if and only if either $a < c$ or 
$\left\{ \begin{matrix}
a = c
\\ b > d
\end{matrix} \right.$. 
Diagonal $(l, l+2)$ from the last proof will have number $1.$ It is clear that the tail of diagonal with number $k$ will be equal to $\lambda_k.$

\begin{proposition}\label{bij}
$\Lambda_{n+2}$ defines a bijection between $T_{n+2}$ and $\mathbb{Y}_n$.
\end{proposition}

\begin{proof}
We begin our proof with the following lemma:

\begin{lemma}
For every $A \in T_{n+2}$ and every $k$ $\lambda_k(A) \leq (n - k).$
\end{lemma}

\begin{proof}
Consider in $A$ diagonal $(a, b)$ number $k.$ It separates our $(n+2)-$gon in two parts, one of them contains vertex 0. $A$ divides this part into triangles by diagonals number $(k + 1), (k + 2), \ldots, (n - 1)$ - $(n - 1 - k)$ diagonals in all. On the other hand this part contains vertices $0, 1, \ldots, a, (n-1)$ and probably some others (if $b < (n + 1))$ - $\quad (a + 2)$ vertices at least. Hence we obtain that $(a - 1) \geq (n - 1 - k) \Leftrightarrow a \geq (n - k).$ Since we know that $a = \lambda_k(A),$ it proves our lemma. 
\end{proof}

Now we see that for every $k$ $\quad \lambda_k(A) + k \leq n,$ so $\Lambda_{n+2}(A)$ defines Young diagram from  $\mathbb{Y}_n.$ By Lemma 1 we obtain that this map $A \mapsto \Lambda_{n+2}(A)$ is an injection. We will prove surjectivity by induction. For $n = 1$ the statement is obvious. Suppose that it holds for $n = k.$ Consider diagram $B \in \mathbb{Y}_{k+1}.$ Let $B^{'}$ be a diagram obtained from $B$ by throwing out the first (the biggest) row which length we will denote by $b_1.$ Clearly $B^{'}$ lies in $\mathbb{Y}_k,$ hence by assumption exists $\Lambda_{k+2}^{-1}(B^{'}) \in T_{k+2}.$ Take a $(k+3)-$gon and its truncation by the line through vertices $b_1$ and $(b_1 + 2)$ (it is a $(k+2)-$gon). It is easy to see that $\Lambda_{k+2}^{-1}(B^{'})$ might be considered as triangulation of this truncation. Therefore $A = (\Lambda_{k+2}^{-1}(B^{'}) \cup \left\{(b_1, b_1 + 2)\right\})$ is a triangulation from $T_{k+3}$ and $\Lambda_{k+3}(A)$ is $B^{'}$ with one glued row of length $b_1$ i. e. $B,$ q.e.d.
\end{proof}

\begin{remark}
In \cite{Stan} (Theorem 6.2.1 and Corollary 6.2.3) natural bijections between $T_{n+2}$ and set of ordered binary trees on $(2n+1)$ vertices, between this set of trees and the set of Dyck paths of length $2n$ are given; in Exercise 6.19 also bijections between set of Dyck paths and some set of sequences and between this set of sequences and $\mathbb{Y}_{n}$ are given. The bijection between $T_{n+2}$ and $\mathbb{Y}_{2n}$ (Exercise 6.19, (a) and (vv)) is constructed as a composition of these bijections thus its explicit view is quite complicated. One can check by right computations that our bijection $\Lambda_{n+2}$ is in fact precisely this composition. Our description is however more simple and very convenient for our further constructions. Quite similar bijection can be found in \cite{Lov}, but it defines in fact $t \circ \Lambda_{n+2} \circ \alpha,$ where $t$ is a transposition of Young diagrams and $\alpha$ is an action of reflection over a perpendicular bisector of the side $(0, n + 1)$ on triangulations (see subsection 2.4 below).
\end{remark}

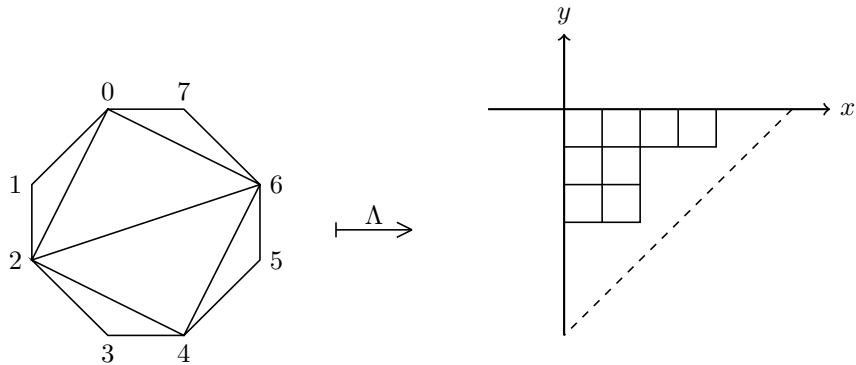
\begin{example}

\begin{tikzpicture}[line width=0.6pt]
  \draw (0,1) node [left] {$2$} -- 
        (0,2) node [left] {$1$} -- 
        (1,3) node [above] {$0$} -- 
        (2,3) node [above] {$7$} -- 
        (3,2) node [right] {$6$} -- 
        (3,1) node [right] {$5$} -- 
        (2,0) node [below] {$4$} -- 
        (1,0) node [below] {$3$} -- 
        (0,1) -- (2,0) -- (3,2) -- (1,3) -- (0,1) -- (3,2);
  
  \draw (4,1.3) -- (4,1.5);
  \draw (4,1.4) -- (5,1.4);
  \draw (4.5,1.6) node {$\Lambda$};
  \draw (5,1.4) -- (4.8,1.5);
  \draw (5,1.4) -- (4.8,1.3);
  
  \draw [<->,thick] (7,4) node (yaxis) [above] {$y$}
     |- (10.5,3) node (xaxis) [right] {$x$};
  \draw [thick] (7,0) -- (7,3);
  \draw [thick] (6,3) -- (7,3);
  \draw (9,3) -- (9,2.5) -- (8,2.5) -- (8,1.5) -- (7,1.5);
  \draw (8.5,3) -- (8.5,2.5);
  \draw (8,3) -- (8,2.5);
  \draw (7.5,1.5) -- (7.5,3);
  \draw (7,2) -- (8,2);
  \draw (7,2.5) -- (8,2.5);
  \draw [dashed] (10,3) -- (7, 0);
\end{tikzpicture}

In this example 
\begin{itemize}
\item[$\bullet$]
$n = 6;$ 
\item[$\bullet$] Triangulation  $A$ is ${(4,6);(2,4);(2,6);(0,2);(0,6)};$ 
\item[$\bullet$] Partition $\Lambda(A)$ is $(4,2,2,0,0,0,\ldots).$
\end{itemize}

\end{example}

\subsection{Flips between Young diagrams}

\begin{definition}
The \textit{flip} between two triangulations is the following operation: one removes a diagonal to create a quadrilateral, then replaces the removed diagonal with the other diagonal of the quadrilateral.
\end{definition}

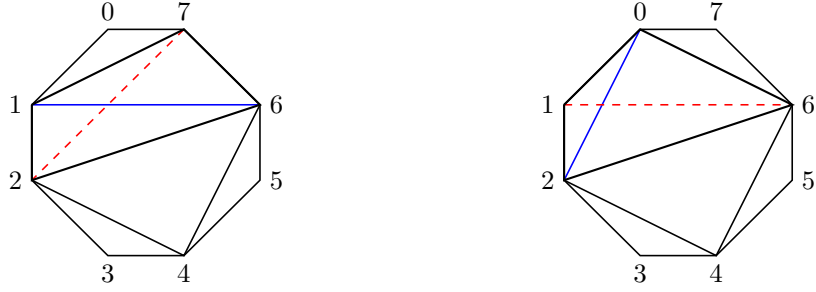
\begin{tikzpicture}[line width=0.6pt]
  \draw (0,1) node [left] {$2$} -- 
        (0,2) node [left] {$1$} -- 
        (1,3) node [above] {$0$} -- 
        (2,3) node [above] {$7$} -- 
        (3,2) node [right] {$6$} -- 
        (3,1) node [right] {$5$} -- 
        (2,0) node [below] {$4$} -- 
        (1,0) node [below] {$3$} -- 
        (0,1);
  \draw (0,1) -- (2,0) -- (3,2);
  \draw [blue] (0,2) -- (3,2);
  \draw [dashed, red] (2,3) -- (0,1);
  \draw [thick] (0,2) -- (0,1) -- (3,2) -- (2,3) -- (0,2);
  
  \draw (7,1) node [left] {$2$} -- 
        (7,2) node [left] {$1$} -- 
        (8,3) node [above] {$0$} -- 
        (9,3) node [above] {$7$} -- 
        (10,2) node [right] {$6$} -- 
        (10,1) node [right] {$5$} -- 
        (9,0) node [below] {$4$} -- 
        (8,0) node [below] {$3$} -- 
        (7,1);
  \draw (7,1) -- (9,0) -- (10,2);
  \draw [blue] (8,3) -- (7,1);
  \draw [dashed, red] (7,2) -- (10,2);
  \draw [thick] (7,2) -- (7,1) -- (10,2) -- (8,3) -- (7,2);
\end{tikzpicture}

\begin{definition} 
We say that there is $n-$flip between two Young diagrams $A, B \in \mathbb{Y}_n$ if there is a flip between $\Lambda_{n+2}^{-1}(A)$ and $\Lambda_{n+2}^{-1}(B)$.
\end{definition}

\begin{proof} {\bf of Theorem \ref{indep}}
Let $l$ be minimal nonnegative number such that $A, B \in \mathbb{Y}_l.$ One can observe that by definition 
$$A = \Lambda_{n+2}(\Lambda_{l+2}^{-1}(A) \cup \left\{ (0, l + 1), (0, l + 2), \ldots, (0, n) \right\}),$$ 
and same for $B,$ therefore if $\Lambda_{n+2}^{-1}(B)$ can be obtained from $\Lambda_{n+2}^{-1}(A)$ by flip then $\Lambda_{l+2}^{-1}(B)$ can be obtained from $\Lambda_{l+2}^{-1}(A)$ by the flip in the same diagonal. It implies that there is $l-$flip between $A$ and $B.$ Surely one can verify the statement for all $k > l$ by similar reasoning.
\end{proof}

Theorem \ref{indep} implies that existence of $n-$flip between two diagrams does not depend on $n,$ hence it is naturally to consider flips instead of $n-$flips. It turns out that we can define flips between diagrams in a simple manner without looking at corresponding triangulations.

\begin{definition} 
Let $M = (\mu_1,\mu_2, \ldots)$ and  $N$ be Young diagrams. We will say that $N$
is obtained from $M$ by a \textit{flip in row $k$}, if we can obtain it from $M$ by throwing out row number $k$  (it can have length $0$) and insertion of another row of length $l$ in such place that it will be Young diagram; where $l$ is defined by the following rule:

Start from the point $(\mu_k, -k)$ of diagram $M$ and go along line $x - y = \mu_k + k.$ If $k$th row is longer than $(k+1)$th, we should go left and downwards, if their lengths are the same - right and upwards. Stop at the first moment, when we touch the boundary of $M$ or the coordinate line. The abscissa of this point will be $l.$
 
We can also define $l$ by the formula: 
\begin{itemize}
\item[1)] Let $\mu_{k+1} = \mu_{k}.$ If $T_k = \left\{m: m < k; m + \mu_m \geq k + \mu_k \right\},$ then $$l = 
\left\{ \begin{matrix}
k + \mu_k - \mbox{max}\left\{T_k\right\}, & T_k\neq\emptyset
\\k + \mu_k,  & T_k=\emptyset
\end{matrix}
\right.$$
\item[2)] Let $\mu_{k+1} < \mu_{k}.$ If $T_k = \left\{m: m > k; m + \mu_m \geq k + \mu_k \right\},$ then $$l = 
\left\{\begin{matrix}
k + \mu_k - \mbox{min}\left\{T_k\right\}, &T_k\neq\emptyset
\\0,  &T_k=\emptyset
\end{matrix}\right.$$
\end{itemize}
\end{definition}

Note that all rows of length $0$ lay in the first case. For triangulations $\Lambda_{t}^{-1}(M)$ this observation corresponds to the fact that an adding of new vertices to our polygon adds new diagonals with tail $0,$ every time with the last number.

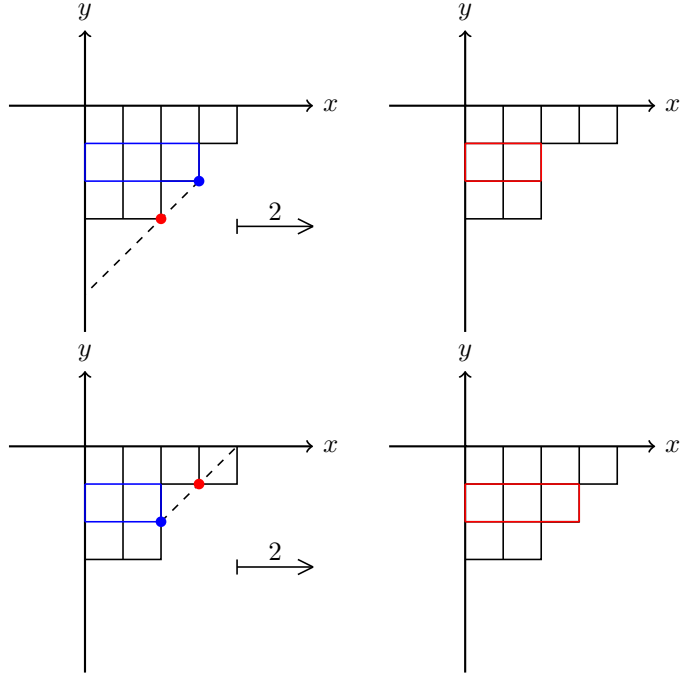
\begin{tikzpicture}[line width=0.6pt]
  \draw [<->,thick] (1,4) node (yaxis) [above] {$y$}
     |- (4,3) node (xaxis) [right] {$x$};
  \draw [thick] (1,0) -- (1,3);
  \draw [thick] (0,3) -- (1,3);
  \draw (3,3) -- (3,2.5) -- (2.5,2.5) -- (2.5,2) -- (2,2) -- (2,1.5) -- (1,1.5);
  \draw (1.5,3) -- (1.5,1.5);
  \draw (2,3) --(2,2);
  \draw (2.5,3) -- (2.5,2.5);
  \draw [blue] (1,2.5) -- (2.5,2.5);
  \draw [blue] (1,2) -- (2.5,2);
  \draw [blue] (1,2.5) -- (1,2);
  \draw [blue] (2.5,2.5) -- (2.5,2);

  %\draw [red] (2.5,2) -- (2,1.5);
  
  \draw [dashed] (2.5,2) -- (1,0.5);
  \fill[blue] (2.5,2) circle (2pt);
  \fill[red] (2,1.5) circle (2pt);

  \draw (3,1.3) -- (3,1.5);
  \draw (3,1.4) -- (4,1.4);
  \draw (3.5,1.6) node {$2$};
  \draw (4,1.4) -- (3.8,1.5);
  \draw (4,1.4) -- (3.8,1.3);
  
  \draw [<->,thick] (6,4) node (yaxis) [above] {$y$}
     |- (8.5,3) node (xaxis) [right] {$x$};
  \draw [thick] (6,0) -- (6,3);
  \draw [thick] (5,3) -- (6,3);
  \draw (8,3) -- (8,2.5) -- (7,2.5) -- (7,1.5) -- (6,1.5);
  \draw (7.5,3) -- (7.5,2.5);
  \draw (7,3) -- (7,2.5);
  \draw (6.5,3) -- (6.5,1.5);
  \draw (6,2) -- (7,2);
  \draw (6,2.5) -- (7,2.5);

  \draw [red] (7,2.5) -- (7,2);
  \draw [red] (6,2) -- (6,2.5);
  \draw [red] (6,2.5) -- (7,2.5);
  \draw [red] (6,2) -- (7,2);
 % \fill[red] (8,2) circle (2pt);
\end{tikzpicture}

\begin{tikzpicture}[line width=0.6pt]
  \draw [<->,thick] (1,4) node (yaxis) [above] {$y$}
     |- (4,3) node (xaxis) [right] {$x$};
  \draw [thick] (1,0) -- (1,3);
  \draw [thick] (0,3) -- (1,3);
  \draw (3,3) -- (3,2.5) -- (2,2.5) -- (2,2) -- (2,1.5) -- (1,1.5);
  \draw (1.5,3) -- (1.5,1.5);
  \draw (2,3) --(2,2);
  \draw (2.5,3) -- (2.5,2.5);
  \draw [blue] (1,2.5) -- (2,2.5);
  \draw [blue] (1,2) -- (2,2);
  \draw [blue] (1,2.5) -- (1,2);
  \draw [blue] (2,2) -- (2,2.5);

  %\draw [red] (2.5,2) -- (2,1.5);
  
  \draw [dashed] (2,2) -- (3,3);
  \fill[blue] (2,2) circle (2pt);
  \fill[red] (2.5,2.5) circle (2pt);

  \draw (3,1.3) -- (3,1.5);
  \draw (3,1.4) -- (4,1.4);
  \draw (3.5,1.6) node {$2$};
  \draw (4,1.4) -- (3.8,1.5);
  \draw (4,1.4) -- (3.8,1.3);
  
  \draw [<->,thick] (6,4) node (yaxis) [above] {$y$}
     |- (8.5,3) node (xaxis) [right] {$x$};
  \draw [thick] (6,0) -- (6,3);
  \draw [thick] (5,3) -- (6,3);
  \draw (8,3) -- (8,2.5) -- (7.5,2.5) -- (7.5,2) -- (7,2) -- (7,1.5) -- (6,1.5);
  \draw (7.5,3) -- (7.5,2.5);
  \draw (7,3) -- (7,2);
  \draw (6.5,3) -- (6.5,1.5);
  \draw (6,2) -- (7.5,2);
  \draw (6,2.5) -- (7.5,2.5);

  \draw [red] (7.5,2.5) -- (7.5,2);
  \draw [red] (6,2) -- (6,2.5);
  \draw [red] (6,2.5) -- (7.5,2.5);
  \draw [red] (6,2) -- (7.5,2);
 % \fill[red] (8,2) circle (2pt);
\end{tikzpicture}

The first picture (above) shows flip in the case of inequality of rows:
$$(M = (4,3,2,0,0,\ldots); k = 2; l = 2; N = (4,2,2,0,0,\ldots));$$
the second one (below) shows flip in the case of equality:
$$(M = (4,2,2,0,0,\ldots); k = 2; l = 3; N = (4,3,2,0,0,\ldots)).$$
One can see that these flips turn out to be inverse to each other.

\begin{theorem} \label{flips}
Some Young diagram can be obtained from another one by flip in some row if and only if there is a $k-$flip between them for each $k$ that $A,B \in \mathbb{Y}_{k}$.
\end{theorem}

\begin{proof}
For the proof we need the following lemma:
\begin{lemma} \label{head}
Consider $D = (d_1, d_2 , \ldots) \in \mathbb{Y}_n$ and its $k-$th row. The last one corresponds to some diagonal of $\Lambda^{-1}(D)$ with one of ends (tail) $d_k.$ Then the second end $l_k$ can be obtained by the following rule:

Start from the point $(d_k, -k)$ of diagram $D$ and go along line $x - y = d_k + k$ right and upwards. Stop at the first moment, when we hit the boundary of $D$ or the coordinate line. Let $m$ be the abscissa of this point,  then $l_k = m + 1.$

We can also define it by formula: $$l_k = 1 + \left(k + d_k - \mbox{max}(\left\{m: m < k; m + d_m > k + d_k \right\} \cup \left\{0\right\}) \right).$$
\end{lemma}
Note that here we need to {\it hit} boundary, it is not enough to stop to {\it touch} boundary.

\begin{proof}
We know already that head of the diagonal number 1 equals $d_1 + 2 = d_1 + 1 + 1.$ It corresponds to the definition of $l_1,$ because a set what we take minimum of is simply $\left\{ 0 \right\}$ (we hit an axle $Ox$). Let us explain why a sequence $\nu_2, \nu_3, \ldots$ of heads of diagonals of $\Lambda^{-1}(D)$ and a sequence $l_2, l_3, \ldots$ can be calculated from $\nu_1 = l_1 = d_1 + 2$ by the same rules. These rules are following:

\begin{itemize}
\item[1)] If $d_{k+1} = d_k$ and $\nu_k = l_k = d_i$ for some $i < k,$ then $\nu_{k+1}$ is equal to the  $\nu_j$; $l_{k+1} = l_{j},$ where $d_j = d_i > d_{j+1}$ (one can observe that equality $\nu_m = l_m$ for all $m \leq k$ implies equality $\nu_{k+1} = l_{k+1}$). 

Suppose now that $\forall i < k: \quad \nu_k \neq d_i.$
\item[2)] If $d_{k+1} = d_k - 1,$ then $\nu_{k+1} = \nu_k, l_{k+1} = l_k.$
\item[3)] If $d_{k+1} \leq d_k - 2,$ then $\nu_{k+1} = d_{k+1} + 2, l_{k+1} = d_{k+1} + 2.$
\item[4)] If $d_{k+1} = d_k,$ then $\nu_{k+1} = \nu_k + 1, l_{k+1} = l_k + 1.$
\end{itemize}

It is clear that there is no other cases. We see also that these formulae imply equality $\nu_{k+1} = l_{k+1}$ for all $k \in \mathbb{N}$ hence they imply the statement of the lemma. Prove them in turn:
\begin{itemize}
\item[1)] At first we prove the equality for $l_{k+1}$ by geometrical approach. We go from the point $(d_k, -k)$ along the line $x - y = d_k + k$ right and upwards until we hit the diagram at some point $(m, h).$ If $(m + 1)$ turns out to be $d_i$ for some $i,$ then $(m + 1, h)$ is one of corners of the diagram. Hence when we go from $(d_{k+1}, -(k+1) = (d_k, -(k+1))$ along the line $x - y = d_k + k + 1$ right and upwards, we encounter the diagram (or the axle $Ox$) at the same point that if we would gone from $(m+1, k).$ This implies required equality for $l_{k+1}.$

For $\nu_{k+1}$ the proof is even more simple: two diagonals $(d_k = d_{k+1}, d_i = \nu_k)$ and $(d_j, \nu_j)$ have common vertex, so there must to be a diagonal $(d_k+1, \nu_j)$ completing them to a triangle. This diagonal has number $(k + 1),$ q.e.d.

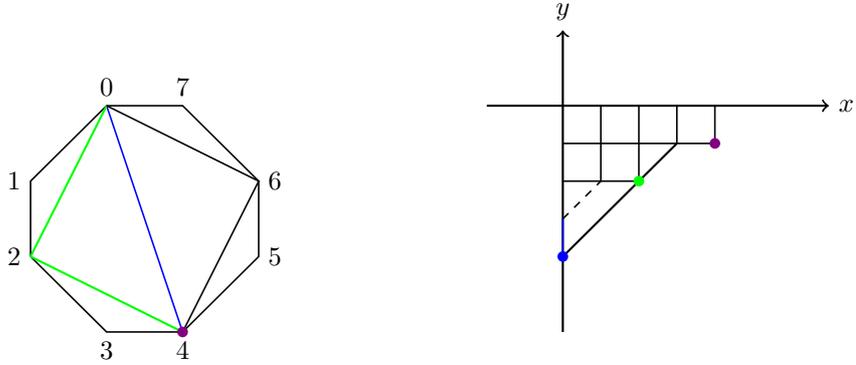
\begin{tikzpicture}[line width=0.6pt]
  \draw (0,1) node [left] {$2$} -- 
        (0,2) node [left] {$1$} -- 
        (1,3) node [above] {$0$} -- 
        (2,3) node [above] {$7$} -- 
        (3,2) node [right] {$6$} -- 
        (3,1) node [right] {$5$} -- 
        (2,0) node [below] {$4$} -- 
        (1,0) node [below] {$3$} -- 
        (0,1);
  \draw [blue] (1,3) -- (2,0);
  \draw [green, thick] (0,1) -- (1,3);
  \draw [green, thick] (0,1) -- (2,0);
  \draw (1,3) --(3,2) -- (2,0);
  \fill [violet] (2,0) circle (2pt);

  \draw [<->,thick] (7,4) node (yaxis) [above] {$y$}
     |- (10.5,3) node (xaxis) [right] {$x$};
  \draw [thick] (7,0) -- (7,3);
  \draw [thick] (6,3) -- (7,3);
  \draw (9,3) -- (9,2.5) -- (8,2.5) -- (8,2) -- (7,2);
  \draw (8.5,3) -- (8.5,2.5);
  \draw (8,3) -- (8,2.5);
  \draw (7.5,3) -- (7.5,2);
  \draw (7,2.5) -- (8,2.5);
  \draw [dashed] (7,1.5) -- (7.5,2);
  \draw [thick] (7,1) -- (8.5,2.5);
  \draw [blue] (7,1) -- (7,1.5);
  \fill[blue] (7,1) circle (2pt);
  \fill[green] (8,2) circle (2pt);
  \fill[violet] (9,2.5) circle (2pt); 
\end{tikzpicture}

\item[2)] By assumption points $(d_k, -k)$ and $(d_{k+1}, -(k + 1))$ both lay on the ine $x - y = d_k + k$ and at the second point the line touches our diagram, hence equality $l_k = l_{k+1}$ is clear. Since there is a triangle in $\Lambda^{-1}(D)$ that has $(d_k, d_{k+1})$ and $(d_k, \nu_k)$ as two of edges, the formula for $\nu_{k+1}$ is obvious. 

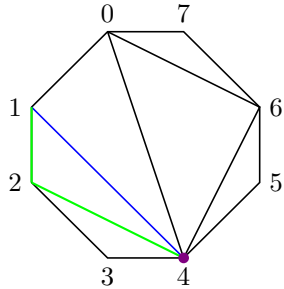
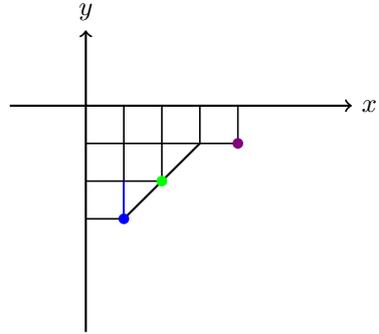
\begin{tikzpicture}[line width=0.6pt]
  \draw (0,1) node [left] {$2$} -- 
        (0,2) node [left] {$1$} -- 
        (1,3) node [above] {$0$} -- 
        (2,3) node [above] {$7$} -- 
        (3,2) node [right] {$6$} -- 
        (3,1) node [right] {$5$} -- 
        (2,0) node [below] {$4$} -- 
        (1,0) node [below] {$3$} -- 
        (0,1);
  \draw [blue] (0,2) -- (2,0);
  \draw [green, thick] (0,1) -- (0,2);
  \draw [green, thick] (0,1) -- (2,0);
  \draw (1,3) --(3,2) -- (2,0) -- (1,3);
  \fill [violet] (2,0) circle (2pt);

  \draw [<->,thick] (7,4) node (yaxis) [above] {$y$}
     |- (10.5,3) node (xaxis) [right] {$x$};
  \draw [thick] (7,0) -- (7,3);
  \draw [thick] (6,3) -- (7,3);
  \draw (9,3) -- (9,2.5) -- (8,2.5) -- (8,2) -- (7,2);
  \draw (8.5,3) -- (8.5,2.5);
  \draw (8,3) -- (8,2.5);
  \draw (7.5,3) -- (7.5,2);
  \draw (7,2.5) -- (8,2.5);
  \draw [thick] (7.5,1.5) -- (8.5,2.5);
  \draw (7,1.5) -- (7.5,1.5) -- (7.5,2);
  \draw [blue] (7.5,1.5) -- (7.5,2);
  \fill[blue] (7.5,1.5) circle (2pt);
  \fill[green] (8,2) circle (2pt);
  \fill[violet] (9,2.5) circle (2pt); 
\end{tikzpicture}

\item[3)] In this case starting from the point $(d_{k+1}, -(k + 1))$ we will encounter the diagram after going through only one square (because $k$-th row is longer than $(k+1)$-th by 2 squares), hence $l_{k+1} = d_{k+1} + 2.$ Surely, as for $1$-th diagonal, for the $(k+1)$-th one the difference between head and tail equals $2,$ that implies $\nu_{k+1} = d_{k+1} + 2$ as we need.

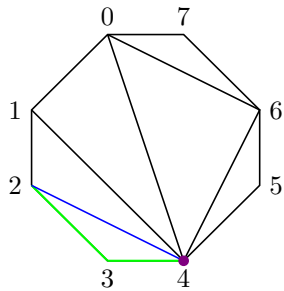
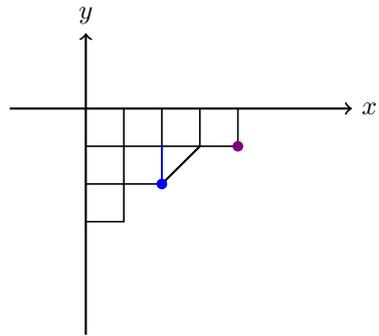
\begin{tikzpicture}[line width=0.6pt]
  \draw (0,1) node [left] {$2$} -- 
        (0,2) node [left] {$1$} -- 
        (1,3) node [above] {$0$} -- 
        (2,3) node [above] {$7$} -- 
        (3,2) node [right] {$6$} -- 
        (3,1) node [right] {$5$} -- 
        (2,0) node [below] {$4$} -- 
        (1,0) node [below] {$3$} -- 
        (0,1);
  \draw (0,2) -- (2,0);
  \draw [green, thick] (0,1) -- (1,0);
  \draw [green, thick] (1,0) -- (2,0);
  \draw [blue] (0,1) -- (2,0);
  \draw (1,3) --(3,2) -- (2,0) -- (1,3);
  \fill [violet] (2,0) circle (2pt);

  \draw [<->,thick] (7,4) node (yaxis) [above] {$y$}
     |- (10.5,3) node (xaxis) [right] {$x$};
  \draw [thick] (7,0) -- (7,3);
  \draw [thick] (6,3) -- (7,3);
  \draw (9,3) -- (9,2.5) -- (8,2.5) -- (8,2) -- (7,2);
  \draw (8.5,3) -- (8.5,2.5);
  \draw (8,3) -- (8,2.5);
  \draw (7.5,3) -- (7.5,2);
  \draw (7,2.5) -- (8,2.5);
  \draw [thick] (8,2) -- (8.5,2.5);
  \draw (7,1.5) -- (7.5,1.5) -- (7.5,2);
  \draw [blue] (8,2) -- (8,2.5);
  \fill [blue] (8,2) circle (2pt);
  \fill [violet] (9,2.5) circle (2pt); 
\end{tikzpicture}

\item[4)] Since the case of touching the diagram by a line $x - y = d_{k+1} + (k+1)$ is observed already at 1), here we will hit $D$ (coming from $(d_{k+1}, -(k + 1))$) at a point with the same ordinate that if we start from $(d_k = d_{k+1}, -k).$ Surely it implies that abscissa will be bigger by $1,$ i.e. $l_{k+1} = l_k + 1.$ 

It is easy to see that in $\Lambda^{-1}(D)$ diagonal number $(k+1)$ is an edge of a triangle with two other edges $(d_k, \nu_k)$ and $(\nu_k, \nu_{k+1})$ so its head is equal to $\nu_k + 1,$ q.e.d. 

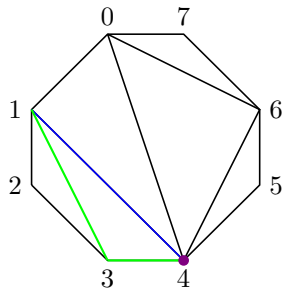
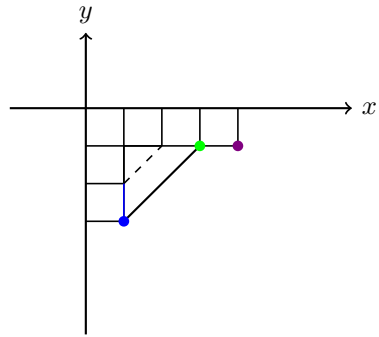
\begin{tikzpicture}[line width=0.6pt]
  \draw (0,1) node [left] {$2$} -- 
        (0,2) node [left] {$1$} -- 
        (1,3) node [above] {$0$} -- 
        (2,3) node [above] {$7$} -- 
        (3,2) node [right] {$6$} -- 
        (3,1) node [right] {$5$} -- 
        (2,0) node [below] {$4$} -- 
        (1,0) node [below] {$3$} -- 
        (0,1);
  \draw (0,2) -- (2,0);
  \draw [blue] (0,2) -- (2,0);
  \draw [green, thick] (1,0) -- (2,0);
  \draw [green, thick] (0,2) -- (1,0);
  \draw (1,3) --(3,2) -- (2,0) -- (1,3);
  \fill [violet] (2,0) circle (2pt);

  \draw [<->,thick] (7,4) node (yaxis) [above] {$y$}
     |- (10.5,3) node (xaxis) [right] {$x$};
  \draw [thick] (7,0) -- (7,3);
  \draw [thick] (6,3) -- (7,3);
  \draw (9,3) -- (9,2.5) -- (7.5,2.5) -- (7.5,1.5) -- (7,1.5);
  \draw (8.5,3) -- (8.5,2.5);
  \draw (8,3) -- (8,2.5);
  \draw (7.5,3) -- (7.5,2);
  \draw (7,2.5) -- (8,2.5);
  \draw (7,2) -- (7.5,2);
  \draw [thick] (7.5,1.5) -- (8.5,2.5);
  \draw [dashed] (7.5,2) -- (8,2.5);
  \fill [green] (8.5,2.5) circle (2pt);
  \draw [blue] (7.5,1.5) -- (7.5,2);
  \fill [blue] (7.5,1.5) circle (2pt);
  \fill [violet] (9,2.5) circle (2pt); 
\end{tikzpicture}

\end{itemize}

\end{proof}

Let us return to the proof of the theorem. Since flip of a triangulation changes one diagonal while flip of a diagram changes one row, we should only prove that a tail of new diagonal coincides with a length of new row. We will prove it for a diagram $D = (d_1, d_2, \ldots)$ and its flip in $k$-th row in three cases independently:
\begin{itemize}
\item[1)]
Let $d_k = d_{k+1} < d_{k-1}.$ Therefore $d_{k-1} + k - 1 > d_k + k - 1 \Rightarrow d_{k-1} + k - 1 \geq d_k + k \Rightarrow \mbox{max}(T_k) = k - 1 \Rightarrow l = d_k + 1.$ Let us understand what can we say about $\Lambda^{-1}(D).$ In this triangulation $k$-th diagonal is the first, but not the last diagonal with tail $d_k.$ Then there are two cases: either tail of $(k-1)-$th diagonal is equal to $(d_k + 1),$ or head of $k-$th one is equal to $(d_k + 2)$ (because of reasons similar to those which we explain in proof of Lemma 2). In the first case (example is on the left figure below) exchange happens in quadrilateral with vertices $d_k, d_{k-1}, \nu_k, \nu_{k+1}$ (here $\nu_i$ denotes head of $i$-th diagonal as in the proof of Lemma 2), so new diagonal is $(d_{k-1} = d_k + 1, \mu_{k+1}).$ These arguments imply equality we are going to prove. In the second case (the right figure below) exchange happens in quadrilateral $(d_k, d_k + 1, \nu_k = d_k + 2, \nu_{k+1} = \nu_{k-1}),$ new diagonal is $(d_k + 1, \nu_{k-1}),$ and we obtain required again, since $\nu_{k-1} > d_{k-1} \geq d_k + 2 > d_k + 1.$

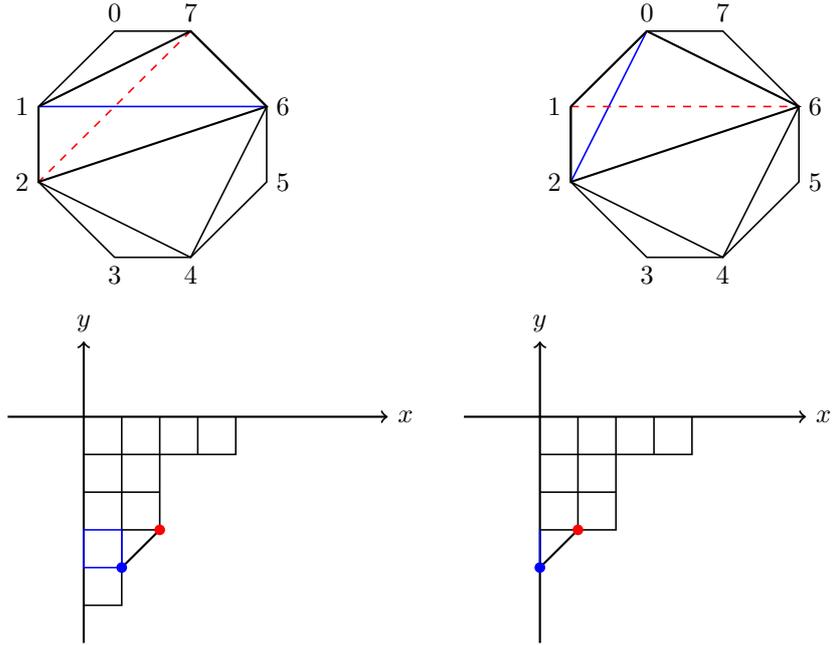
\begin{tikzpicture}[line width=0.6pt]
  \draw (0,1) node [left] {$2$} -- 
        (0,2) node [left] {$1$} -- 
        (1,3) node [above] {$0$} -- 
        (2,3) node [above] {$7$} -- 
        (3,2) node [right] {$6$} -- 
        (3,1) node [right] {$5$} -- 
        (2,0) node [below] {$4$} -- 
        (1,0) node [below] {$3$} -- 
        (0,1);
  \draw (0,1) -- (2,0) -- (3,2);
  \draw [blue] (0,2) -- (3,2);
  \draw [dashed, red] (2,3) -- (0,1);
  \draw [thick] (0,2) -- (0,1) -- (3,2) -- (2,3) -- (0,2);
  
  \draw (7,1) node [left] {$2$} -- 
        (7,2) node [left] {$1$} -- 
        (8,3) node [above] {$0$} -- 
        (9,3) node [above] {$7$} -- 
        (10,2) node [right] {$6$} -- 
        (10,1) node [right] {$5$} -- 
        (9,0) node [below] {$4$} -- 
        (8,0) node [below] {$3$} -- 
        (7,1);
  \draw (7,1) -- (9,0) -- (10,2);
  \draw [blue] (8,3) -- (7,1);
  \draw [dashed, red] (7,2) -- (10,2);
  \draw [thick] (7,2) -- (7,1) -- (10,2) -- (8,3) -- (7,2);
\end{tikzpicture}

\begin{tikzpicture}[line width=0.6pt]
  \draw [<->,thick] (1,4) node (yaxis) [above] {$y$}
     |- (5,3) node (xaxis) [right] {$x$};
  \draw [thick] (1,0) -- (1,3);
  \draw [thick] (0,3) -- (1,3);
  \draw (3,3) -- (3,2.5) -- (2,2.5) -- (2,1.5) -- (1.5,1.5) -- (1.5,0.5) -- (1,0.5);
  \draw (1.5,3) -- (1.5,1.5);
  \draw (2,3) --(2,2.5);
  \draw (2.5,3) -- (2.5,2.5);
  \draw (1,2.5) -- (2,2.5);
  \draw (1,2) -- (2,2);
  \draw [blue] (1,1.5) -- (1,1) -- (1.5,1) -- (1.5,1.5) -- (1,1.5) ;
  \draw [thick] (1.5,1) -- (2,1.5);
  \fill[blue] (1.5,1) circle (2pt);
  \fill[red] (2,1.5) circle (2pt);

  \draw [<->,thick] (7,4) node (yaxis) [above] {$y$}
     |- (10.5,3) node (xaxis) [right] {$x$};
  \draw [thick] (7,0) -- (7,3);
  \draw [thick] (6,3) -- (7,3);
  \draw (9,3) -- (9,2.5) -- (8,2.5) -- (8,1.5) -- (7,1.5);
  \draw (8.5,3) -- (8.5,2.5);
  \draw (8,3) -- (8,2.5);
  \draw (7.5,3) -- (7.5,1.5);
  \draw (7,2) -- (8,2);
  \draw (7,2.5) -- (8,2.5);
  \draw [blue] (7,1) -- (7,1.5);
  \draw [thick] (7,1) -- (7.5,1.5);
  \fill [blue] (7,1) circle (2pt);
  \fill [red] (7.5,1.5) circle (2pt);
 
 % \draw [red] (8,2.5) -- (8,2);
  %\draw [red] (7,2) -- (7,2.5);
  %\draw [red] (7,2.5) -- (8,2.5);
  %\draw [red] (7,2) -- (8,2);
 % \fill[red] (8,2) circle (2pt);
\end{tikzpicture}

\item[2)]
Let $d_k = d_{k+1} = d_{k-1}.$ Then it is easy to observe that exchange will happen in quadrilateral $(d_k, \nu_{k-1}, \nu_k, \nu_{k+1}),$ therefore the head of new diagonal will equal $\nu_{k-1}$ (figure below). On the other hand this new head is equal, by definition on flips of diagrams, to  $$d_k + k - \mbox{max}(\left\{m: m < k; m + d_m \geq k + d_k \right\} \cup \left\{0\right\}) =$$ $$= d_{k-1} + (k - 1) + 1 - \mbox{max}(\left\{m: m < (k-1); m + d_m > (k - 1) + d_{k-1} \right\} \cup \left\{0\right\}) = \nu_{k-1},$$ where the last one equality follows from Lemma 2. We obtain required equality.

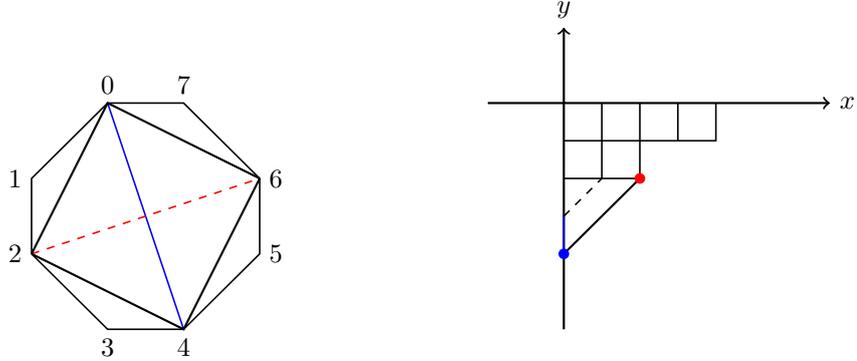
\begin{tikzpicture}[line width=0.6pt]
  \draw (0,1) node [left] {$2$} -- 
        (0,2) node [left] {$1$} -- 
        (1,3) node [above] {$0$} -- 
        (2,3) node [above] {$7$} -- 
        (3,2) node [right] {$6$} -- 
        (3,1) node [right] {$5$} -- 
        (2,0) node [below] {$4$} -- 
        (1,0) node [below] {$3$} -- 
        (0,1);
  \draw [blue] (1,3) -- (2,0);
  \draw [dashed, red] (0,1) -- (3,2);
  \draw [thick] (0,1) -- (1,3) -- (3,2);
  \draw [thick] (0,1) -- (2,0) -- (3,2);

  \draw [<->,thick] (7,4) node (yaxis) [above] {$y$}
     |- (10.5,3) node (xaxis) [right] {$x$};
  \draw [thick] (7,0) -- (7,3);
  \draw [thick] (6,3) -- (7,3);
  \draw (9,3) -- (9,2.5) -- (8,2.5) -- (8,2) -- (7,2);
  \draw (8.5,3) -- (8.5,2.5);
  \draw (8,3) -- (8,2.5);
  \draw (7.5,3) -- (7.5,2);
  \draw (7,2.5) -- (8,2.5);
  \draw [dashed] (7,1.5) -- (7.5,2);
  \draw [thick] (7,1) -- (8,2);
  \draw [blue] (7,1) -- (7,1.5);
  \fill[blue] (7,1) circle (2pt);
  \fill[red] (8,2) circle (2pt); 
\end{tikzpicture}

\item[3)] 
Let $d_k > d_{k+1}.$ Then for a flip of a triangulation we will change a diagonal in a quadrilateral with vertices $d_k, t, \nu_k, d_l,$ where $d_l > d_k > t > \nu_k \quad$ and $(d_l, \nu_k)$ is the first diagonal (it has number $l$) after $k$-th with head $\nu_k,$ hence a new diagonal will have $d_l$ as a tail. On the other hand by geometrical formulation of Lemma 2 $(d_l, -l)$ is the first point of the diagram $D$ on the way from $(d_k, -k)$ along the line $x - y = d_k + k$ left and downwards. By geometrical definition of flip between diagrams we obtain required result.

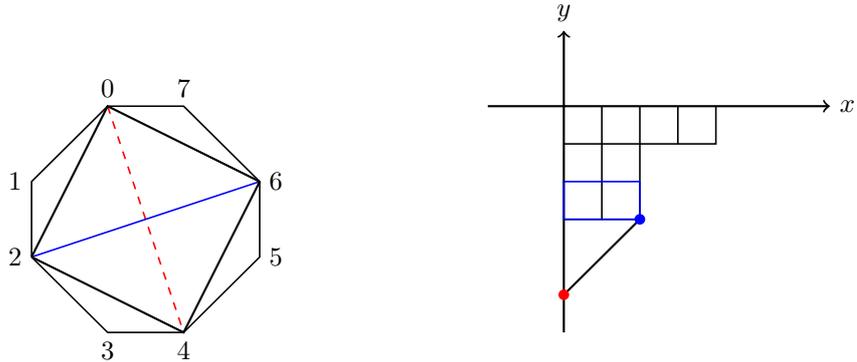
\begin{tikzpicture}[line width=0.6pt]
  \draw (0,1) node [left] {$2$} -- 
        (0,2) node [left] {$1$} -- 	
        (1,3) node [above] {$0$} -- 
        (2,3) node [above] {$7$} -- 
        (3,2) node [right] {$6$} -- 
        (3,1) node [right] {$5$} -- 
        (2,0) node [below] {$4$} -- 
        (1,0) node [below] {$3$} -- 
        (0,1);
  \draw [dashed, red] (1,3) -- (2,0);
  \draw [blue] (0,1) -- (3,2);
  \draw [thick] (0,1) -- (1,3) -- (3,2);
  \draw [thick] (0,1) -- (2,0) -- (3,2);

  \draw [<->,thick] (7,4) node (yaxis) [above] {$y$}
     |- (10.5,3) node (xaxis) [right] {$x$};
  \draw [thick] (7,0) -- (7,3);
  \draw [thick] (6,3) -- (7,3);
  \draw (9,3) -- (9,2.5) -- (8,2.5) -- (8,1.5) -- (7,1.5);
  \draw (8.5,3) -- (8.5,2.5);
  \draw (8,3) -- (8,2.5);
  \draw (7.5,3) -- (7.5,1.5);
  \draw (7,2.5) -- (8,2.5);
  \draw [thick] (7,0.5) -- (8,1.5);
  \draw [blue] (8,2) -- (8,1.5) -- (7,1.5) -- (7,2) -- (8,2);
  \fill[red] (7,0.5) circle (2pt);
  \fill[blue] (8,1.5) circle (2pt); 
\end{tikzpicture}

\end{itemize} 
This examination of cases completes our proof.
\end{proof}

\begin{problem}
Describe an action of flips on Schur polynomials corresponding to Young diagrams.
\end{problem}

\subsection{Infinite-dimensional associahedron}

In this section we will discuss associahedra - some well-known polytopes arising in a number of combinatorial problems. They can be described in very different ways. We will use the most convenient for us accordingly to \cite{Lee}.

\begin{definition} 
\textit{Partial triangulation} of polygon is subset of some triangulation. We define \textit{associahedron} (or \textit{Stasheff polytope}) $As^{n}$ of dimension $n$ as the combinatorial polytope by the following rules:
\begin{itemize}
\item[1)] $k$-dimensional faces of this polytopes are enumerated by partial triangulations of $(n+3)$-gon with $(n-k)$ diagonals. For example, vertices of associahedron correspond to triangulations of $(n+3)$-gon while facets correspond to diagonals.
\item[2)] Face $A$ is contained in a face $B$ if and only if partial triangulation corresponding to $A$ is a subset of partial triangulation corresponding to $B$. 
\end{itemize}
\end{definition}

There are different geometrical realizations of this combinatorial polytope, one can see, e.g., \cite{CZ}. However, we are interested only in combinatorial structure. By Proposition \ref{bij} vertices of $As^n$  correspond bijectively to Young diagrams from $\mathbb{Y}_{n+1}.$ Let us describe flips between diagrams in terms of associahedra. Two vertices of associahedron $As^n$ are connected by an edge when there is some partial triangulation with $(n-1)$ diagonals that is subset of both triangulations $A, B \in T_{n+3}$ corresponding to these vertices. By definition of flips between triangulations the last condition is equivalent to existence of flip between $A$ and $B.$ By Theorems \ref{indep} and \ref{flips} we obtain that two vertices of associahedron are connected by edge if and only if there is a flip between corresponding Young diagrams. We have described $1$-skeleton of associahedron in terms of Young diagrams. Following facts mean that we can describe in these terms all combinatorial structure of this polytope.

\begin{definition}
Combinatorial $n$-dimensional polytope is \textit{simple} if each its vertex belongs to exactly $n$ facets, or, equivalently, each its vertex belongs to exactly $n$ edges.
\end{definition}

It is known that $As^n$ is a simple polytope. Indeed, each triangulation $A \in T_{n+3}$ has exactly $n$ subsets of cardinality $(n-1)$ (we can throw out one of $n$ diagonals), hence corresponding vertex of $As^n$ belongs to exactly $n$ edges. 
\begin{theorem} {\normalfont{(Blind-Mani \cite{BM}, Kalai \cite{Kal})}}
Combinatorial simple polytope is determined uniquely by its $1$-skeleton.
\end{theorem}

\begin{corollary} \label{1}
$As^n$ is determined uniquely by its $1$-skeleton for all $n.$
\end{corollary}

We see that sets $\mathbb{Y}_n$ are connected with assocaihedra but these sets are not very natural objects; in representation theory more useful restrictions on Young diagrams than the line $y - x = n$ are a number of squares or a number of rows. By  definition one can observe that there is natural inclusion $As^{n} \hookrightarrow As^{n+1}$ as a facet: for each partial triangulation $A$ of $(n+3)$-gon $A \cup (0, (n+2))$ is partial triangulation of $(n+4)$-gon; corresponding faces of $As^n$ and $As^{n+1}$ have the same dimension; inclusions of partial triangulations stay the same (modulo union with $(0, (n+2))$) hence inclusions of faces of $As^n$ correspond bijectively to inclusions of faces $As^{n+1}$ lying in the facet corresponding to the diagonal $(0, (n+2)).$
A sequence of these inclusions yields a filtration 
\begin{equation} \label{filt}
As^{0} \hookrightarrow As^{1} \hookrightarrow As^2 \hookrightarrow \ldots 
\end{equation}
The following definition rises naturally from this filtration.

\begin{definition}
We call by \textit{infinite-dimensional associahedron $As^{\infty}$} a direct limit of filtration (\ref{filt})
\end{definition}

\begin{corollary}
$As^{\infty}$'s combinatorial structure is determined uniquely by filtration of $1$-skeletons of $As^n$ arising from (\ref{filt}); thus it is determined by flip operators on the set of all Young diagrams and a filtration
$$As^{0} \hookrightarrow As^{1} \hookrightarrow As^2 \hookrightarrow \ldots $$
\end{corollary}

\subsection{Action of group $D_{n+2}$ on $\mathbb{Y}_n$}

The dihedral group $D_{n+2}$ of symmetries of the right $(n+2)$-gon acts in natural way on the set $T_{n+2}$: symmetry acts on each diagonal of a triangulation while polygon stay unchanged. As an example we consider an action of two generators of $D_{n+2},$ reflection $\alpha$ over a perpendicular bisector of the side $(0,n+1)$ and a rotation $\beta$ by $\frac{2\pi}{n+2}$ counter-clockwise, on some triangulation of right $8$-gon (i.e. in the case $n = 6$):

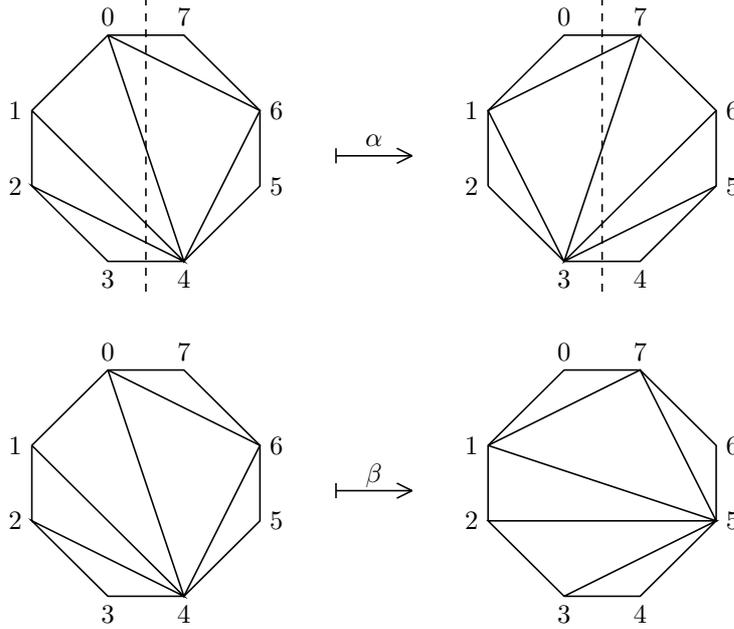
\begin{tikzpicture}[line width=0.6pt]
  \draw (0,1) node [left] {$2$} -- 
        (0,2) node [left] {$1$} -- 
        (1,3) node [above] {$0$} -- 
        (2,3) node [above] {$7$} -- 
        (3,2) node [right] {$6$} -- 
        (3,1) node [right] {$5$} -- 
        (2,0) node [below] {$4$} -- 
        (1,0) node [below] {$3$} -- 
        (0,1) -- (2,0) -- (0,2);
  \draw (1,3) --(3,2) -- (2,0) -- (1,3);
  \draw[dashed] (1.5,3.5) -- (1.5,-0.5);      
  
  \draw (4,1.3) -- (4,1.5);
  \draw (4,1.4) -- (5,1.4);
  \draw (4.5,1.6) node {$\alpha$};
  \draw (5,1.4) -- (4.8,1.5);
  \draw (5,1.4) -- (4.8,1.3);

  \draw (6,1) node [left] {$2$} -- 
        (6,2) node [left] {$1$} -- 
        (7,3) node [above] {$0$} -- 
        (8,3) node [above] {$7$} -- 
        (9,2) node [right] {$6$} -- 
        (9,1) node [right] {$5$} -- 
        (8,0) node [below] {$4$} -- 
        (7,0) node [below] {$3$} -- 
        (6,1);
  \draw (7,0) -- (6,2) -- (8,3) -- (7,0) -- (9,2);
  \draw (7,0) -- (9,1);
  \draw[dashed]  (7.5,3.5) -- (7.5,-0.5);

\end{tikzpicture}
\newline

\begin{tikzpicture}[line width=0.6pt]
  \draw (0,1) node [left] {$2$} -- 
        (0,2) node [left] {$1$} -- 
        (1,3) node [above] {$0$} -- 
        (2,3) node [above] {$7$} -- 
        (3,2) node [right] {$6$} -- 
        (3,1) node [right] {$5$} -- 
        (2,0) node [below] {$4$} -- 
        (1,0) node [below] {$3$} -- 
        (0,1) -- (2,0) -- (0,2);
  \draw (1,3) --(3,2) -- (2,0) -- (1,3);
  
  \draw (4,1.3) -- (4,1.5);
  \draw (4,1.4) -- (5,1.4);
  \draw (4.5,1.6) node {$\beta$};
  \draw (5,1.4) -- (4.8,1.5);
  \draw (5,1.4) -- (4.8,1.3);

  \draw (6,1) node [left] {$2$} -- 
        (6,2) node [left] {$1$} -- 
        (7,3) node [above] {$0$} -- 
        (8,3) node [above] {$7$} -- 
        (9,2) node [right] {$6$} -- 
        (9,1) node [right] {$5$} -- 
        (8,0) node [below] {$4$} -- 
        (7,0) node [below] {$3$} -- 
        (6,1) -- (9,1) -- (7,0);
  \draw (6,2) -- (8,3) -- (9,1) -- (6,2);
        
\end{tikzpicture}

It is absolutely clear that this action on $T_{n+2}$ commutes with flips, hence we can say that elements of $D_{n+2}$ define symmetries of $1$-skeleton of $As^{n-1}.$ By Corollary \ref{1} it means that this action of $D_{n+2}$ defines symmetries of the entire $As^{n-1}$ (as a combinatorial object). 

By Proposition \ref{bij} we obtain that we can define an action of $D_{n+2}$ on $\mathbb{Y}_n$ commuting with flips (as a composition of above action with the map $\Lambda_{n+2}$). Below are given descriptions of actions of $\alpha$ and $\beta$ on $\mathbb{Y}_n$ in terms of diagrams; for $\beta$ there is a geometrical description too.

\begin{proposition} \label{alpha}
Let $A = (a_1, a_2, \ldots)$ be a diagram from $\mathbb{Y}_n.$ Let $(l_1, l_2, \ldots, l_{n-1})$ be a sequence of heads of diagonals of triangulation $\Lambda_{n+2}^{-1}(A),$ defined in Lemma \ref{head}. Then $$\alpha A = (n + 1 - l_{n-1}, \ldots, n + 1 - l_2, n + 1 - l_1, 0,0, \ldots).$$ 
\end{proposition}

\begin{proof}
It is easy to observe that by action of $\alpha$ on triangulation $\Lambda_{n+2}^{-1}(A)$ a diagonal $(a_k, l_k)$ maps to diagonal $(n + 1 - l_k, n + 1 - a_k).$ Hence in view of Lemma \ref{head} the required formula follows immediately.
\end{proof}

\begin{proposition} \label{beta}
Let $A$ be a diagram from $\mathbb{Y}_n.$ Then $\beta A$ is defined as follows: we add one square to each row of $A$ from $1$-th to $(n-1)$-th, then we throw out all rows that intersect the line $y = x - n$ and move above all remaining rows in unique way to obtain a diagram from $\mathbb{Y}_n.$ This new diagram is $\beta A.$ Equivalently, if $A$ consists of rows of lengths $a_i,$ such that $a_i \leq (n-i),$ then $\beta A$ consists of rows of lengths $$\beta a_i = 
\left\{\begin{matrix}
a_i + 1, & \quad a_i < n - i;
\\ 0, & \quad a_i = n - i \quad \mbox{or} \quad i > n - 1,
\end{matrix}\right.$$
decreasingly ordered.
\end{proposition}

\begin{proof}
An equivalence of definitions from the statement is obvious, hence prove only the second one. Consider a diagonal $(a_k, l_k)$ of a triangulation $\Lambda_{n+2}^{-1}(A).$ If $l_k < (n + 1),$ then by action of $\beta$ it passes to a diagonal $(a_{k+1}, l_{k+1}),$ otherwise, i.e. if $l_k = (n + 1),$ it passes to $(0, a_k + 1).$ By definition of $\Lambda_{n+2}$ it is clear that it is enough to show only that this division into cases corresponds to the one from the second definition from the statement, i.e. that $(a_i = n - i) \Leftrightarrow (l_k = (n+1)).$ But $(a_i = n - i) \Leftrightarrow (a_i + i = n),$ and the last one equation is equivalent to $(l_k = (n+1))$ by Lemma \ref{head}.
\end{proof}

%\begin{problem}
%Describe an action of $D_{n+2}$ on Schur polynomials corresponding to the diagrams from $\mathbb{Y}_n.$
%\end{problem}

\begin{remark}
It would be intersting to study possible links between symmetries of $\mathbb{Y}_n$ defined by action of $D_{n+2}$ and some problems devoted to $q,t$-Catalan numbers, such as so called symmetry problem (\cite{Hag}, Open Problem 3.11). One can easily formulate how does area statistics change under this action, while analogous question for bounce and dinv statistics seems less clear.
\end{remark} 

\section{Cluster algebras of type $A_{\infty}$}

In this section we will consider cluster algebras possibly connected with our constructions. All general definitions we will formulate according to \cite{Kel}.

\subsection{Cluster algebras without coefficients}

\begin{definition}
Let us recall that a \textit{quiver} $Q$ is an oriented graph, in other words it is a quadruple given
by a set $Q_0$ (the set of vertices), a set $Q_1$ (the set of \textit{arrows}, or oriented edges) and two maps $s : Q_1 \rightarrow Q_0$ and $t : Q_1 \rightarrow Q_0$ which take an arrow to its source respectively its target.
\end{definition}

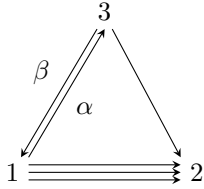
\begin{tikzpicture}
 \draw[->,>=stealth] (1,2) -- (3,2);
 \draw[->,>=stealth] (1,2.1) node [left] (a) {1} --  (3,2.1) node [right] (b) {2} ;
 \draw[->,>=stealth] (1,2.2) -- (3,2.2);
 \draw[->,>=stealth] (1,2.3) -- node [below right] {$\alpha$} (2,4) node [above] {3};
 \draw[->,>=stealth] (1.9,4) -- node [above left] {$\beta$} (0.9,2.3);
 \draw[->,>=stealth] (2.1,4) -- (3,2.3);
\end{tikzpicture}

A \textit{loop} is an arrow whose source coincides with its target; \textit{$2$-cycle} is a pair of distinct arrows $\alpha \neq \beta$ such that $s(\alpha) = t(\beta)$ and $s(\beta) = t(\alpha).$ Quiver $Q$ is \textit{finite} if both sets $Q_0$ and $Q_1$ are finite.

\begin{definition}
Let us fix $n \in \mathbb{N}.$ We call by \textit{seed} a pair $(R, u),$ where
\begin{itemize}
\item[$\bullet$] $R$ is a finite quiver without loops or $2$-cycles with vertex set $Q_0 = \left\{1, 2, \ldots, n \right\};$
\item[$\bullet$] $u$ is a free generating set $\left\{u_1,\ldots, u_n\right\}$ of the field $\mathbb{Q}(x_1,\ldots, x_n)$ of fractions of the polynomial ring $\mathbb{Q}[x_1,\ldots, x_n]$ in $n$ indeterminates.
\end{itemize}

Since $R$ does not have $2$-cycles all arrows from $R_1$ between any two given vertices point in the same
direction. Let $(R, u)$ be a seed and $k$ a vertex of $R.$ The mutation
$\mu_k(R, u)$ of $(R, u)$ at $k$ is the seed $(R', u'),$ where
\begin{itemize}
\item[a)] $R'$ is obtained from $R$ as follows:
\begin{itemize}
\item[1)] reverse all arrows incident with $k;$
\item[2)] for all vertices $i \neq j$ distinct from $k,$ modify the number of arrows between $i$ and $j$ as
follows:
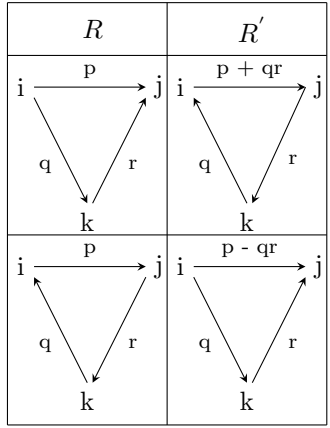
\begin{center}
\begin{tikzpicture} [scale=0.7]
 \draw[->,>=stealth] (1,5.4) node [left] {i} -- node [above] {{\footnotesize{p}}} (3.1,5.4) node [right] {j};
 \draw[->,>=stealth] (1,5.2) -- node [below left] {{\footnotesize{q}}} (2,3.2) node [below] {k};
 \draw[->,>=stealth] (2.1,3.2) -- node [below right] {{\footnotesize{r}}} (3.1,5.2);

 \draw[->,>=stealth] (4,5.4) node [left] {i} -- node [above] {{\footnotesize{p + qr}}} (6.1,5.4) node [right] {j};
 \draw[->,>=stealth] (6.1,5.4) -- node [below right] {{\footnotesize{r}}} (5.1,3.2);
 \draw[->,>=stealth] (5,3.2) node [below] {k} -- node [below left] {{\footnotesize{q}}} (4,5.2);

 \draw[->,>=stealth] (1,2) node [left] {i} -- node [above] {{\footnotesize{p}}} (3.1,2) node [right] {j};
 \draw[->,>=stealth] (3.1,1.8) -- node [below right] {{\footnotesize{r}}} (2.1,-0.2);
 \draw[->,>=stealth] (2,-0.2) node [below] {k} -- node [below left] {{\footnotesize{q}}} (1,1.8);

 \draw[->,>=stealth] (4,2) node [left] {i} -- node [above] {{\footnotesize{p - qr}}} (6.1,2) node [right] {j};
 \draw[->,>=stealth] (4,1.8) -- node [below left] {{\footnotesize{q}}} (5,-0.2) node [below] {k};
 \draw[->,>=stealth] (5.1,-0.2) -- node [below right] {{\footnotesize{r}}} (6.1,1.8);

 \draw (3.5,-1) -- (3.5, 7);
 \draw (0.5,-1) -- (0.5, 7);
 \draw (6.5,-1) -- (6.5, 7);
 
 \draw (0.5,-1) -- (6.5,-1);
 \draw (0.5,7) -- (6.5,7);
 \draw (0.5,6) -- (6.5,6);
 \draw (0.5,2.6) -- (6.5,2.6);
 
 \draw (2.1,6.5) node {$R$};
 \draw (5.1,6.5) node {$R^{'}$};
\end{tikzpicture}
\end{center}
where $p, q, r$ are non negative integers, an arrow $i \stackrel{l}{\rightarrow}
 j$ with $l > 0$ means that $l$ arrows
go from $i$ to $j$ and an arrow $i \stackrel{l}{\rightarrow}
 j$ with $l < 0$ means that $-l$ arrows go from $j$ to $i.$
\end{itemize}
\item[b)] $u'$ is obtained from $u$ by replacing the element $u_k$ with
\begin{equation} \label{er}
u_k^{'} = \frac{1}{u_k} \left( \prod\limits_{\mbox{\footnotesize{arrow}} \quad i \rightarrow k} u_i + \prod\limits_{\mbox{\footnotesize{arrow}} \quad k \rightarrow j} u_j \right).
\end{equation}
\end{itemize}

In the exchange relation (\ref{er}), if there are no arrows from $i$ to $k,$ the product is taken
over the empty set and equals $1.$ It is not hard to see that $\mu_k(R, u)$ is indeed a seed and that $\mu_k$ is
an involution.

Let $Q$ be a finite quiver without loops or $2$-cycles with vertex set $\left\{1,\ldots, n \right\}.$ Consider the \textit{initial seed $(Q, x)$} consisting of $Q$ and the set $x$ formed by the
variables $x_1,\ldots, x_n.$ We define
\begin{itemize}
\item{} the \textit{clusters with respect to $Q$} to be the sets $u$ appearing in seeds $(R, u)$ obtained from $(Q, x)$
by iterated mutation,
\item{} the \textit{cluster variables} for $Q$ to be the elements of all clusters,
\item{} the \textit{cluster algebra} $\mathcal{A}_Q$ to be the $\mathbb{Q}$-subalgebra of the field $\mathbb{Q}(x_1,\ldots, x_n)$ generated by all cluster variables.
\item{} The \textit{exchange graph} associated with $Q$ to be the graph whose vertices are the seeds modulo simultaneous renumbering of the vertices and the associated cluster variables and whose edges correspond to
mutations.
\end{itemize}
\end{definition}

Exchange graphs are characterized by the following theorem:

\begin{theorem} {\normalfont{(Gekhtman-Shapiro-Vainshtein, \cite{GSV})}} \label{geshva}
For cluster algebras associated with quivers following statements hold:
\begin{itemize}
\item[1)]
Every seed is uniquely defined by its cluster; thus, the vertices
of the exchange graph can be identified with the clusters, up to a permutation of
cluster variables.
\item[2)]
Two clusters are adjacent in the exchange graph if and only if they differ in 
exactly $1$ cluster variable.
\end{itemize}
\end{theorem}

\begin{definition}
A \textit{cluster algebra of type $A_n$} (without coefficients) is an algebra that has as the quiver in initial seed an orientation of Dynkin diagram $A_n$. The one such quiver we will use is the following one:
\begin{center}
\begin{tikzpicture}

  \draw[->,>=stealth] (0,0) node [below] {1} -- (0.95,0);
  \draw[->,>=stealth] (1,0) node [below] {2} -- (2,0);
  \draw (2.2,0) node {.};
  \draw (2.4,0) node {.};
  \draw (2.6,0) node {.};
  \draw[->,>=stealth] (2.8,0) -- (3.75,0) node [below] {\large{n}};
  \fill[black] (1,0) circle (2pt);
  \fill[black] (0,0) circle (2pt);
  \fill[black] (3.8,0) circle (2pt);
\end{tikzpicture}
\end{center}
\end{definition}

Let us define infinite analogue of that:

\begin{definition}
We will call by \textit{cluster algebra of type $A_\infty$} (without coefficients) an algebra that has the following quiver that we will denote by $\vec{A}_{\infty}$ in initial seed:
\begin{center}
\begin{tikzpicture}

  \draw[->,>=stealth] (0,0) node [below] {1} -- (0.95,0);
  \draw[->,>=stealth] (1,0) node [below] {2} -- (2,0);
  \draw (2.2,0) node {.};
  \draw (2.4,0) node {.};
  \draw (2.6,0) node {.};
  \draw[->,>=stealth] (2.8,0) -- (3.75,0) node [below] {\large{n}};
  \draw[->,>=stealth] (3.8,0) -- (4.75,0) node [below] {\large{n}\normalsize{+1}};
  \draw[->,>=stealth] (4.8,0) -- (5.8,0);
  \fill[black] (1,0) circle (2pt);
  \fill[black] (0,0) circle (2pt);
  \fill[black] (3.8,0) circle (2pt);
  \fill[black] (4.8,0) circle (2pt);
  \draw (6,0) node {.};
  \draw (6.2,0) node {.};
  \draw (6.4,0) node {.};
  \draw (6.6,0) node {.};
\end{tikzpicture}
\end{center}
We will work with $\mathbb{Q}(x_1, x_2,\ldots)$ instead of $\mathbb{Q}(x_1,\ldots,x_n),$ and $x = \left\{x_1,x_2,\ldots,x_n,\ldots\right\}$ will form initial cluster with above quiver.

By clusters we will consider all sets $u$ from seeds $(R,u),$ obtained from the initial seed by finite number of mutations.
\end{definition}

We can define in the same manner cluster algebras of types $B_{\infty}$, $C_{\infty}$ and $D_{\infty},$ e.g. initial quiver of type $D_{\infty}$-algebra will be the following one:

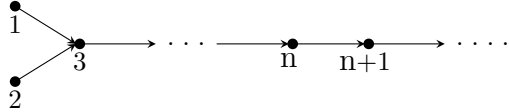
\begin{center}
\begin{tikzpicture}

  \draw[->,>=stealth] (0.15,0.5) node [below] {1} -- (0.95,0);
  \draw[->,>=stealth] (0.15,-0.5) node [below] {2} -- (0.95,0);
  \draw[->,>=stealth] (1,0) node [below] {3} -- (2,0);
  \draw (2.2,0) node {.};
  \draw (2.4,0) node {.};
  \draw (2.6,0) node {.};
  \draw[->,>=stealth] (2.8,0) -- (3.75,0) node [below] {\large{n}};
  \draw[->,>=stealth] (3.8,0) -- (4.75,0) node [below] {\large{n}\normalsize{+1}};
  \draw[->,>=stealth] (4.8,0) -- (5.8,0);
  \fill[black] (1,0) circle (2pt);
  \fill[black] (0.15,-0.5) circle (2pt);
  \fill[black] (0.15,0.5) circle (2pt);
  \fill[black] (3.8,0) circle (2pt);
  \fill[black] (4.8,0) circle (2pt);
  \draw (6,0) node {.};
  \draw (6.2,0) node {.};
  \draw (6.4,0) node {.};
  \draw (6.6,0) node {.};
\end{tikzpicture}
\end{center}

However, associahedra and all constructions from Section 2 are related to algebras of type $A$ only, hence we will consider them mainly.

\subsection{Cluster algebras with coefficients}

\begin{definition}
Let $1 \leq n \leq m$ be integers. An \textit{ice quiver} of type $(n,m)$ is a quiver $\widetilde{Q}$ with a  vertex
set
$$\left\{1, \ldots,m\right\} = \left\{1, \ldots, n \right\} \cup \left\{n + 1, \ldots,m \right\}$$
such that there are no arrows between any vertices $i, j$ which are strictly greater than $n.$ The
\textit{principal part} of $\widetilde{Q}$ is the full subquiver $Q$ of $\widetilde{Q}$ whose vertex set is $\left\{1, \ldots, n\right\}$ (a subquiver is full
if, with any two vertices, it contains all the arrows between them). The vertices $n + 1, \ldots, m$ are
often called \textit{frozen} vertices. The cluster algebra
$$\mathcal{A}_{\widetilde{Q}} \subset \mathbb{Q}(x_1, \ldots, x_m)$$
is defined as before but
\begin{itemize}
\item{} only mutations with respect to vertices in the principal part are allowed and no arrows are
drawn between the vertices greater than $n$,
\item{} in a cluster
$$u = \left\{u_1, \ldots, u_n, c_{n+1}, \ldots, c_m\right\}$$
only $u_1, \ldots, u_n$ are called cluster variables; the elements $c_{n+1}, \ldots , c_m$ are called \textit{coefficients};
to make things clear, the set $u$ is often called an \textit{extended cluster};
\item{} the cluster type of $\widetilde{Q}$ is that of $Q$ if it is defined.
\end{itemize}
\end{definition}

For the type $A_n$ we can reformulate our definitions in terms of the triangulations. We may do it for cluster algebras without coefficients, but let us consider a more general case: we will define a \textit{cluster algebra of type $A_n$} as a cluster algebra whose initial quiver's principal part is an orientation of the Dynkin diagram $A_n.$ For some choice of frozen vertices the language of triangulations is the most convenient.
Assume that some triangulation of the $(n + 3)-$gon determines an initial seed for the cluster algebra and
hence an ice quiver $\widetilde{Q}$ whose frozen vertices correspond to the sides of the $(n+3)-$gon and whose non frozen vertices  - to the diagonals in the triangulation. The arrows of the quiver are determined by
the exchange relations which appear when we wish to make flip of the triangulation. It is not hard to see that this means that the underlying graph of $\widetilde{Q}$ is the
graph dual to the triangulation and that the orientation of the edges of this graph is induced by
the choice of an orientation of the plane. Here is an example of a triangulation and the associated
ice quiver:

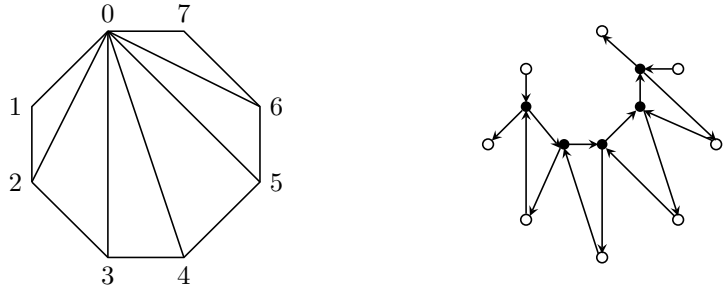
\begin{tikzpicture}[line width=0.6pt]
  \draw (0,1) node [left] {$2$} -- 
        (0,2) node [left] {$1$} -- 
        (1,3) node [above] {$0$} -- 
        (2,3) node [above] {$7$} -- 
        (3,2) node [right] {$6$} -- 
        (3,1) node [right] {$5$} -- 
        (2,0) node [below] {$4$} -- 
        (1,0) node [below] {$3$} -- 
        (0,1);
  \draw (1,3) -- (0,1); 
  \draw (1,3) -- (3,2);
  \draw (1,3) -- (3,1);
  \draw (1,3) -- (1,0);
  \draw (1,3) -- (2,0);

  \draw[->,>=stealth] (6.5,2.43) -- (6.5,2.05);
  \draw[->,>=stealth] (6.5,2) -- (6.95,1.45);
  \draw[->,>=stealth] (7,1.5) -- (7.45,1.5);
  \draw[->,>=stealth] (7.5,1.5) -- (7.95,1.95);
  \draw[->,>=stealth] (8,2) -- (8,2.45);
  \draw[->,>=stealth] (8,2.5) -- (7.5,2.95);
 
  %\draw[->,>=stealth] (6,1.57) -- (6.43,2.5);
  \draw[->,>=stealth] (6.45,1.95) -- (6.05,1.55);
  \draw[->,>=stealth] (6.5,0.58) -- (6.5,1.95);
  \draw[->,>=stealth] (6.95,1.45) -- (6.55,0.55);
  \draw[->,>=stealth] (7.45,0.05) -- (7,1.45);
  \draw[->,>=stealth] (7.5,1.45) -- (7.5,0.05);
  \draw[->,>=stealth] (8.45,0.55) -- (7.55,1.45);
  \draw[->,>=stealth] (8.05,1.95) -- (8.5,0.55);
  \draw[->,>=stealth] (8.95,1.55) -- (8.05,1.95);
  \draw[->,>=stealth] (8.05,2.45) -- (9,1.55);
  \draw[->,>=stealth] (8.42,2.5) -- (8.05,2.5);
  %\draw[->,>=stealth] (7.55,3) -- (8.5,2.55);

  \fill[black] (6.5,2) circle (2pt);
  \fill[black] (7,1.5) circle (2pt);
  \fill[black] (7.5,1.5) circle (2pt);
  \fill[black] (8,2) circle (2pt);
  \fill[black] (8,2.5) circle (2pt);
  
  \draw (6.5,2.5) circle (2pt);
  \draw (6,1.5) circle (2pt);
  \draw (6.5,0.5) circle (2pt);
  \draw (7.5,0) circle (2pt);
  \draw (8.5,0.5) circle (2pt);
  \draw (9,1.5) circle (2pt);
  \draw (8.5,2.5) circle (2pt);
  \draw (7.5,3) circle (2pt);
  
\end{tikzpicture}

It is not hard to verify that the algebra defined above is actually a cluster algebra of type $A_n$ with $(n+3)$ coefficients. Since this algebra does not depend to the triangulation that we started from, we may say that the initial ice quiver of an algebra of type $A_n$ with $(n+3)$ coefficients that we will consider is the following one:

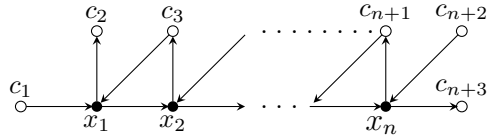
\begin{center}
\begin{tikzpicture}

  \draw[->,>=stealth] (0.07,0) -- (0.95,0);
  \draw[->,>=stealth] (1,0) node [below] {$x_1$} -- (1.95,0);
  \draw[->,>=stealth] (2,0) node [below] {$x_2$} -- (2.95,0);
  
  \draw (3.2,0) node {.};
  \draw (3.4,0) node {.};
  \draw (3.6,0) node {.};
  
  \draw (3.2,1) node {.};
  \draw (3.4,1) node {.};
  \draw (3.6,1) node {.};
  \draw (3.8,1) node {.};
  \draw (4,1) node {.};
  \draw (4.2,1) node {.};
  \draw (4.4,1) node {.};
  \draw (4.6,1) node {.};
  
  \draw[->,>=stealth] (3.8,0) -- (4.75,0) node [below] {\large{$x_n$}};
  \draw[->,>=stealth] (4.8,0) -- (5.75,0);
  
  \draw[->,>=stealth] (1,0.07) -- (1,0.93);
  \draw[->,>=stealth] (2,0.07) -- (2,0.93);
  \draw[->,>=stealth] (4.8,0.07) -- (4.8,0.93);
  
  \draw[->,>=stealth] (1.95,0.95) -- (1.05,0.05);
  \draw[->,>=stealth] (2.95,0.95) -- (2.05,0.05);
  \draw[->,>=stealth] (4.75,0.95) -- (3.85,0.05);
  \draw[->,>=stealth] (5.75,0.95) -- (4.85,0.05);
  
  \draw(0,0) node [above] {$c_1$} circle (2pt);
  \draw(1,1) node [above] {$c_2$} circle (2pt);
  \draw(2,1) node [above] {$c_3$} circle (2pt);
  \draw(4.8,1) node [above] {$c_{n+1}$} circle (2pt);
  \draw(5.8,1) node [above] {$c_{n+2}$} circle (2pt);
  \draw(5.8,0) node [above] {$c_{n+3}$} circle (2pt);
  
  \fill[black] (1,0) circle (2pt);
  \fill[black] (2,0) circle (2pt);
  \fill[black] (4.8,0) circle (2pt);
\end{tikzpicture}
\end{center}

Now we define infinite analogue of that:

\begin{definition}

We will call by a {\it cluster algebra of type $\widetilde{A}_{\infty}$} a cluster algebra which initial quiver is the following one:

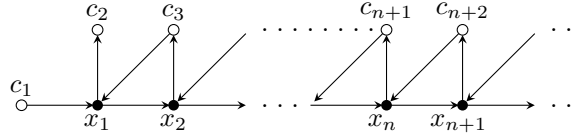
\begin{center}
\begin{tikzpicture}

  \draw[->,>=stealth] (0.07,0) -- (0.95,0);
  \draw[->,>=stealth] (1,0) node [below] {$x_1$} -- (1.95,0);
  \draw[->,>=stealth] (2,0) node [below] {$x_2$} -- (2.95,0);
  
  \draw (3.2,0) node {.};
  \draw (3.4,0) node {.};
  \draw (3.6,0) node {.};
  
  \draw (3.2,1) node {.};
  \draw (3.4,1) node {.};
  \draw (3.6,1) node {.};
  \draw (3.8,1) node {.};
  \draw (4,1) node {.};
  \draw (4.2,1) node {.};
  \draw (4.4,1) node {.};
  \draw (4.6,1) node {.};
  
  \draw[->,>=stealth] (3.8,0) -- (4.75,0) node [below] {$x_n$};
  \draw[->,>=stealth] (4.8,0) -- (5.75,0) node [below] {$x_{n+1}$};
  \draw[->,>=stealth] (5.8,0) -- (6.75,0);

  \draw[->,>=stealth] (1,0.07) -- (1,0.93);
  \draw[->,>=stealth] (2,0.07) -- (2,0.93);
  \draw[->,>=stealth] (4.8,0.07) -- (4.8,0.93);
  \draw[->,>=stealth] (5.8,0.07) -- (5.8,0.93);
  
  \draw[->,>=stealth] (1.95,0.95) -- (1.05,0.05);
  \draw[->,>=stealth] (2.95,0.95) -- (2.05,0.05);
  \draw[->,>=stealth] (4.75,0.95) -- (3.85,0.05);
  \draw[->,>=stealth] (5.75,0.95) -- (4.85,0.05);
  \draw[->,>=stealth] (6.75,0.95) -- (5.85,0.05);
  \draw (7,0) node {.};
  \draw (7.2,0) node {.};
  \draw (7.4,0) node {.};
  
  \draw (7,1) node {.};
  \draw (7.2,1) node {.};
  \draw (7.4,1) node {.};
  
  \draw(0,0) node [above] {$c_1$} circle (2pt);
  \draw(1,1) node [above] {$c_2$} circle (2pt);
  \draw(2,1) node [above] {$c_3$} circle (2pt);
  \draw(4.8,1) node [above] {$c_{n+1}$} circle (2pt);
  \draw(5.8,1) node [above] {$c_{n+2}$} circle (2pt);
  
  \fill[black] (1,0) circle (2pt);
  \fill[black] (2,0) circle (2pt);
  \fill[black] (4.8,0) circle (2pt);
  \fill[black] (5.8,0) circle (2pt);
\end{tikzpicture}
\end{center}

We will work with $\mathbb{Q}(x_1, x_2,\ldots)$ instead of $\mathbb{Q}(x_1,\ldots,x_n),$ and $x = \left\{x_1,x_2,\ldots,x_n,\ldots\,c_1,c_2,\ldots,c_n,\ldots\right\}$ 
%(where elements are algebraically independent) 
will form initial seed with above quiver. 

By extended clusters we will consider all sets $u$ from seeds $(R,u),$ obtained from the initial seed by finite number of mutations; the set of cluster variables we define as the set of all images of $x_i$ under these mutations.
\end{definition}

Since seed mutations correspond to flips of triangulations, it is clear that the exchange graph of the cluster algebra of type $A_n$ defined above is a $1-$skeleton of $As^n.$ Similarly, the exchange graph of the cluster algebra of type $\widetilde{A}_{\infty}$ is a $1-$skeleton of $As^{\infty}.$ One can easily check that if we forget about coefficients (one can consider them to be $1$) these exchange graphs will not change, hence exchange graphs of $A_n$-type and $A_{\infty}$-type cluster algebras without coefficients are $1-$skeletons of $As^n$ and $As^{\infty}$ respectively too.

Using construction from the Section 2 we obtain a bijection between the set of Young diagrams and the set of clusters of an $A_{\infty}$-type cluster algebra (and the set of clusters of an $\widetilde{A}_{\infty}$-type cluster algebra), whose restriction on $\mathbb{Y}_n$ provides the following statement:

\begin{proposition}
There is a one-to-one correspondence between $\mathbb{Y}_n$ and the set of clusters of a cluster algebra of type $A_n.$
\end{proposition}

Since rows of a diagram which lie above the line $y = - (n+1)$ correspond to diagonals of a triangulation, they correspond also to cluster variables (it would not be a bijection from the set of all rows of all diagrams to the set of all cluster variables, but it would be a one-to-one correspondence between rows of each diagram and variables of associated cluster). In some sense, one may say that the columns of a diagram correspond to the sides of of our polygon, but it should be stipulated that a length of a column is not equal to any end of a side. A length of the column between lines $x = k$ and $x = (k - 1)$ is equal surely to the number of rows of length greater or equal to $k;$ therefore (by our bijection $\Lambda$) to the number of diagonals of associated triangulation whose tail is greater than $(k-1),$ in other words, lying totally "at the right side" of the side $(k-1, k).$ Anyway, one may associate with each side of the polygon, except for the side $(n+2, 0),$ some column of a diagram (or, at least, a vertical stripe whose part it is). We obtain a bijection for each diagram between its columns and frozen variables (except for one) of corresponding extended cluster. Note that only columns "above" the line $y = x - (n + 1)$ are involved.

\begin{remark} 
One knows that the algebra generated by Pl{\"u}cker coordinates and Pl{\"u}cker relations for the Grassmanian  $G(2, n+3)$ (equivalently, the algebra of polynomial functions on the cone over this Grassmanian) has the structure of a cluster algebra of type $A_n$ with $(n+3)$ coefficients described above (\cite{FZ2}), whose (extended) clusters are some bases of this algebra. One checks that if we take $\mathbb{C}^{\infty}$ as a direct limit of the filtration $\mathbb{C}^n \hookrightarrow \mathbb{C}^{n+1}$ and consider an analogous algebra, we will obtain a cluster algebra of type $\widetilde{A}_{\infty}.$
\end{remark}

\subsection{Transposition of diagrams as a symmetry inside seeds}

On the set of Young diagrams there exists a well-known and natural symmetry  - transposition. Considered on the set $\mathbb{Y}_n,$ it would define a symmetry $t$ on the set of triangulations $T_{n+2}.$ It gives rise to the following problem:

\begin{problem}
Give explicit combinatorial description of the symmetry $t$, without applying to the language of diagrams.
\end{problem}

Let us try to understand what the transposition of diagrams corresponds to for the algebra $\widetilde{A}_{\infty}$ defined in the previous subsection. On the one hand, it provides a symmetry on the set of vertices of $As^{\infty}$, therefore a symmetry on the set of seeds of $\widetilde{A}_{\infty}$. This symmetry is not a symmetry of the associahedron, i.e. does not save edges between vertices, in other words, it does not commute with flips. Nevertheless, it provides a symmetry on the set of bases of the algebra from the remark at the end of previous subsection, which are clusters. On the other side, we can consider the transposition from a different point of view: we remember that rows of an arbitrary diagram correspond to cluster variables while columns correspond to frozen ones. If we consider all rows and all columns, not only those which lie above some line, we come exactly to the situation with countable sets of cluster and frozen variables. Our algebra $\widetilde{A}_{\infty}$ is suitable for such an approach also because in the associated picture with triangulations only one side goes from the vertice $0,$ while for every algebra of type $A_n$ we face the problem of an absence at a diagram of a column corresponding to the side $(n+2,0).$ With this approach the transposition changing cluster to frozen variables and vice versa will do it inside each seed. It is clear that numbers of edges between vertices of quivers will change, and mutations will not stay unchanged. However, thanks to the equality of cardinalities of sets of cluster and frozen variables inside each seed, we will obtain some analogue of a seed, corresponding to the same vertice of $As^{\infty}$ as before.

\begin{problem}
Define strictly the object obtained from $\widetilde{A}_{\infty}$ by an exchange, corresponding to the transposition of diagrams, of cluster to frozen variables and vice versa. 
\end{problem}

This exchange of variables might be more natural for quantum cluster algebras introduced by Berenstein and Zelevinsky at \cite{BZ}. For quantum cluster algebras of type $A_n$ an exchange graph coincides with an exchange graph for simple cluster algebras of this type, thus all relations to Young diagrams would 
hold.

\subsection{Quiver representations and cluster algebra of type $A_{\infty}$}

\begin{definition}
Let $Q$ be a finite quiver without oriented cycles. For example, $Q$ can be an orientation of a
simply laced Dynkin diagram or the quiver
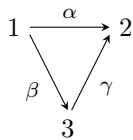
\begin{center}
\begin{tikzpicture}[scale=0.5]
 
 \draw[->,>=stealth] (1,5.4) node [left] {1} -- node [above] {{\footnotesize{$\alpha$}}} (3.1,5.4) node [right] {2};
 \draw[->,>=stealth] (1,5.2) -- node [below left] {{\footnotesize{$\beta$}}} (2,3.2) node [below] {3};
 \draw[->,>=stealth] (2.1,3.2) -- node [below right] {{\footnotesize{$\gamma$}}} (3.1,5.2);

\end{tikzpicture}
\end{center}
Let $k$ be an algebraically closed field. A \textit{representation} of $Q$ is a diagram of finite-dimensional
vector spaces of the shape given by $Q.$ More formally, a representation of $Q$ is the datum $V$ of
\begin{itemize}
\item{} a finite-dimensional vector space $V_i$ over $k$ for each vertex $i$ of $Q,$
\item{} a linear map $V_\alpha : V_i \rightarrow V_j$ for each arrow $\alpha : i \rightarrow j$ from $Q_1.$
\end{itemize}
Thus, in the above example, a representation of $Q$ is a (not necessarily commutative) diagram
\begin{center}
\begin{tikzpicture}[scale=0.5]
 
 \draw[->,>=stealth] (1,5.4) node [left] {$V_1$} -- node [above] {{\footnotesize{$V_{\alpha}$}}} (3.1,5.4) node [right] {$V_2$};
 \draw[->,>=stealth] (1,5.2) -- node [below left] {{\footnotesize{$V_{\beta}$}}} (2,3.2) node [below] {$V_3$};
 \draw[->,>=stealth] (2.1,3.2) -- node [below right] {{\footnotesize{$V_{\gamma}$}}} (3.1,5.2);

\end{tikzpicture}
\end{center}
formed by three finite-dimensional vector spaces and three linear maps.

A \textit{subrepresentation} $V'$
of a representation $V$ is given by a family of subspaces $V'_i \subset V_i, i \in Q_0$, such that the image of 
$V'_i$ under $V_\alpha$ is contained in $V'_j$ for each arrow $\alpha : i \rightarrow j$ from $Q_1.$

A \textit{dimension vector} of representation $V$ is a sequence $\underline{\mbox{dim}} V$ of dimensions $\mbox{dim}V_i, i \in Q_0.$

A \textit{direct sum} $V \oplus W$ of two given representations is the representation given by
$$ (V \oplus W)_i = V_i \oplus W_i \quad \mbox{and} \quad (V \oplus W)_{\alpha} = V_{\alpha} \oplus W_{\alpha},$$
for all $i \in Q_0$ and $\alpha \in Q_1.$

A representation $V$ is \textit{indecomposable} if it is non zero and in each decomposition $V = V' \oplus V''$
we have $V' = 0$ or $V'' = 0.$

A quiver is called \textit{representation-finite} if it has only finitely many isomorphism classes of indecomposable representations.
\end{definition}

For cluster algebras of finite type and quivers from their initial seeds following statements hold:

\begin{theorem} {\normalfont{(Fomin-Zelevinsky \cite{FZ2})}} \label{FZ}
\textit{Let $Q$ be a finite connected quiver without loops or $2$-cycles
with vertex set $\left\{1,\ldots, n \right\}.$ Let $\mathcal{A}_Q$ be the associated cluster algebra.}
\begin{itemize}
\item[a)] All cluster variables are Laurent polynomials, i.e. their denominators are monomials. In each such Laurent polynomial, the coefficients in the numerator are positive integers.
\item[b)] The number of cluster variables is finite if and only if $Q$ is mutation equivalent to an orientation of a simply laced Dynkin diagram $\Delta.$ In this case, $\Delta$ is unique and the non initial cluster variables are in bijection with the positive roots of $\Delta;$ namely, if we denote the simple
roots by $\alpha_1,\ldots, \alpha_n,$ then for each positive root $\sum d_i \alpha_i$ there is a unique non initial cluster variable whose denominator is $\prod x_i^{d_i}.$
\end{itemize}
\end{theorem}

\begin{theorem} \label{Gab} {\normalfont{(Gabriel \cite{G})}}. Let $Q$ be a connected quiver and assume that $k$ is algebraically closed. $Q$ is representation-finite if and only if the underlying graph of $Q$ is a simply laced Dynkin diagram $\Delta.$ In this case the map taking a representation with dimension vector $(d_i)$ to a root $\sum d_i \alpha_i$ of the root system associated with $\Delta$ yields a bijection from the set of isomorphism classes of indecomposable representations to the set of positive roots.
\end{theorem}

\begin{corollary} \label{cor}
The map taking an indecomposable representation $V$ with dimension vector {\normalfont{$\underline{\mbox{dim}} V = (d_i)$}} of a representation-finite quiver $Q$ to the unique non initial cluster variable $X_V$ with denominator $\prod x_i^{d_i}$ yields a bijection from the set of isomorphism classes of indecomposable representations to the set of non initial cluster variables.
\end{corollary}

Define

\begin{equation} \label{CCV}
CC(V) = \frac{1}{x_1^{d_1} x_2^{d_2} \ldots x_n^{d_n}}\left( \sum\limits_{0 \leq e \leq d} \chi(Gr_e(V))\prod\limits_{i=1}^{n} x_i^{\sum_{j \rightarrow i} e_i + \sum_{i \rightarrow j} (d_j - e_j)}\right).
\end{equation}

Here the sum is taken over all vectors $e \in \mathbb{N}^n$ such that $0 \leq e_i \leq d_i$ for all $i.$ For vector $e,$ the \textit{quiver Grassmanian} $Gr_e(V)$ is the variety of $n$-tuples of subspaces $U_i \subset V_i$ such that $\mbox{dim} U_i = e_i$ and $U_i$ form a subrepresentation of $V.$ One can check (\cite{CC}) that $Gr_e(V)$ identifies with a projective subvariety of $\prod\limits_{i=1}^{n} Gr_{e_i}(V_i)$ (the product of ordinary Grassmanians). Some restrictions on the Euler characteristic $\chi$ can be found in \cite{CC} or \cite{Kel}; in the case $k = \mathbb{C}$ it is taken with respect to singular cohomologies with coefficients in arbitrary field (e.g. $\mathbb{Q}$).

\begin{theorem} {\normalfont{(Caldero-Chapoton \cite{CC})}} \label{C-C}
Let $Q$ be a Dynkin quiver and $V$ an indecomposable representation.
Then we have $CC(V) = X_V.$
\end{theorem}

We will call by root system $A_\infty$ a system each root of which is a finite linear combination of simple roots of system $A_n$ for a sufficiently big $n.$ Equivalently, each root is a root of system $A_n$ for a sufficiently big $n.$ 

We are ready to formulate and prove the main new result of this subsection:

\begin{proposition} \label{infan}
For the cluster algebra $A_\infty$ without coefficients (and the quiver $\vec{A}_{\infty}$)  statements of Theorems \ref{geshva}, \ref{FZ}, \ref{Gab} and \ref{C-C} and Corollary \ref{cor} holds. This result holds for cluster algebras $B_{\infty},$ $C_{\infty}$ and $D_{\infty}$ as well.
\end{proposition}

\begin{proof}
Let us prove statements of Theorem \ref{geshva}. Consider arbitrary cluster $u$ lying in a seed $(R,u).$ Let $n \in \mathbb{N}$ be a maximal number for which $x_{n-1} \notin u.$ If we consider now restriction $R^{'}$ of a quiver $R$ on the set of vertices $\left\{1,\ldots,n\right\}$ and put $u^{'} = u \backslash \left\{x_{n+1},x_{n+2},\ldots\right\},$ then $(R^{'},u^{'})$ will be a seed of an algebra of type $A_n$ (obtained from the initial one by the same sequence of mutations that maps the initial seed of $A_{\infty}$) to $(R,u).$ Since the first statement of the theorem holds for this algebra, a seed $(R^{'},u^{'})$ is determined uniquely, therefore $(R,u)$ is determined uniquely too (since we know all arrows of $R$ for those at least one vertice is bigger than $n,$ - all of them are $k \rightarrow (k+1)$ for $k \geq n).$

Now prove the second statement. We know that the exchange graph of an algebra of type $A_{\infty}$ is the $1-$skeleton of $As^{\infty}.$ Hence each two of its vertices belong both to 
$1-$skeleton of $As^n$ for sufficiently big $n.$ This means that they are connected by an edge iff corresponding two clusters $v$ and $w$ of an algebra of type $A_n$ differ in exactly one cluster variable. Since corresponding to our vertices clusters of $A_{\infty}-$type algebra have forms $v \cup \left\{x_{n+1},x_{n+2},\ldots\right\}$ and $w \cup \left\{x_{n+1},x_{n+2},\ldots\right\}$ respectively, the statement is proved.

All other statements follow from the fact that each (non-initial) cluster variable appears at the some finite step, i.e. it is a cluster variable of $A_n$-type algebra for sufficiently big $n.$ Thus it corresponds by bijections from Theorems \ref{FZ} and \ref{Gab} to unique (up to isomorphism) indecomposable representation of the quiver from the initial seed of $A_n$-type algebra (and to unique positive root of the system $A_n$ being simultaneously a root of the system $A_\infty$). Adding to this representation zero vector spaces $V_k$ and maps $V_{k-1} \rightarrow V_k, k > n,$ we obtain an indecomposable representation of the quiver from the initial seed of $A_{\infty}$-type algebra. There are no other representations and roots corresponding to our variable, since one can use similar arguments backwards. Hence the statements of Theorems \ref{FZ} and \ref{Gab} and Corollary \ref{cor} are proved for our case. Since addition of zero spaces and maps changes neither Euler characteristics, no products in the right hand side of the formula (\ref{CCV}), the statement of the Theorem \ref{C-C} holds in our case too, q.e.d.
\end{proof}

\subsection{Cluster algebra and cluster category of type $A_{\infty}$}

Our construction of the cluster algebra of type $A_{\infty}$ is different than a construction of a cluster
category $\mathscr{D}$ with the Auslander-Reiten quiver $\mathbb{Z} A_{\infty}$ given by Holm and J{\o}rgensen and investigated in \cite{HJ} and \cite{JP}, but there is quite transparent connection between them. To explain this we will use notation from Sections 2 and 6 of \cite{JP}. There exists a bijection, which we will denote by $\phi,$ from the set of indecomposable objects of $\mathscr{D}$ to the set of arcs $(m,n)$ connecting non-neighboring integers. We consider a cluster tilting subcategory $\mathscr{T}$ and write $T = \mbox{ind} \mathscr{T}.$ Then $T$ bijectively corresponds by $\phi$ to a collection $\mathfrak{T}$ of arcs. Theorems 4.3. and 4.4 of \cite{HJ} shows that the fact that $\mathscr{T}$ is cluster tilting is equivalent to the fact that $\mathfrak{T}$ is a maximal collection of non-crossing arcs which is {\it locally finite} or has a {\it fountain}. 
%The second case is interesting for us. 
Having a fountain means that there is an integer $n$ such that $\mathfrak{T}$ contains infinitely many arcs of the form $(m,n)$ and infinitely many of the form $(n,p).$ Being locally finite means that for all $n$ $\mathfrak{T}$ contains only finitely many arcs of the form $(m,n)$ and only finitely many of the form $(n,p).$

Consider $\mathscr{E} = \mbox{add} E$ where $E$ is the set of indecomposable objects of $\mathscr{D}$ which can be reached by finitely many mutations from $\mathscr{T}.$ In the equivalent language of arcs it means that $\phi E$ is the set of arcs which can be obtained as elements of images of $\mathfrak{T}$ by finitely many flips. Here by {\it flip} we mean the similar operation than above: one removes one of arcs in a maximal non-crossing collection and replaces it by the unique another to obtain a maximal collection of non-crossing arcs again. 

\begin{theorem} {\normalfont (\cite{JP})} \label{obj}
\begin{itemize}
\item[(i)] If $\mathfrak{T}$ is locally finite, then $E = \mbox{ind} \mathscr{D}.$

\item[(ii)] If $\mathfrak{T}$ has a fountain at $n$, then $E$ is the set of all objects corresponding to arcs of form $(k,l)$ where $k < l \leq n$ or $n \leq k < l.$ 
\end{itemize}
\end{theorem}

%$\mathfrak{T_0} = \left\{(n-i-1,n);(n,n+i+1)|i \in \mathbb{N} \right\}$. 
We will use the following definition from \cite{BIRS}:

\begin{definition}
A map
$$\rho : \mbox{obj} \mathscr{E} \rightarrow \mathbb{Q}(x_z)_{z\in \mathfrak{T}}$$
is called a cluster map if it satisfies the following conditions:
\begin{itemize}
\item[(i)] $\rho$ is constant on isomorphism classes.
\item[(ii)] If $c_1, c_2 \in \mathscr{E}$ then $\rho(c_1 \oplus c_2) = \rho(c_1)\rho(c_2).$
\item[(iii)] If $m, l  \in \mbox{obj} \mathscr{E}$ are indecomposable objects with $\mbox{dim} \mbox{Ext}^1_{\mathscr{D}} (m, l) = 1$ and
$b, b' \in \mbox{obj} \mathscr{E}$ are such that there are non-split distinguished triangles
$$m \rightarrow b \rightarrow l; \quad l \rightarrow b' \rightarrow m$$
in $\mathscr{D},$ then $\rho(m)\rho(l) = \rho(b) + \rho(b').$
\item[(iv)] There is a cluster tilting subcategory $\mathscr{T'}$ which can be reached from $\mathscr{T}$ for
which $\left\{ \rho(t') | t' \in \mbox{ind} \mathscr{T'} \right\}$ is a transcendence basis of the field $\mathbb{Q}(x_z)_{z \in \mathfrak{T}}.$
\end{itemize}
\end{definition}

Consider an arbitrary $\mathfrak{T}$ having a fountain at $n$ and a field $\mathbb{Q}(x_t, y_{t'}),$ where $t \in \left\{(i, j) \in  \mathfrak{T}, i < j \leq n \right\},$  $t' = \left\{(i', j') \in \mathfrak{T}, n \leq i' < j' \right\}.$

\begin{theorem} {\normalfont(\cite{JP})} \label{fount} 
There exists cluster map $\rho^{\mathscr{T}} : \mbox{obj} \mathscr{E} \to \mathbb{Q}(x_t, y_t')$ with following properties:

\begin{itemize}
\item[(i)] $\rho^{\mathscr{T}} (\phi^{-1}(t)) = x_t, \rho^{\mathscr{T}}(\phi^{-1}(t')) = y_{t'}$ for all $t, t'.$
\item[(ii)] If $\phi^{-1}((k,l)) \in \mathscr{E},$ then $\rho^{\mathscr{T}}(\phi^{-1}((k,l)))$ is a  non-zero Laurent polynomial (in $x_t$ and $y_{t'}$). The coefficients in the numerator of this polynomial are positive integers.
\item[(iii)] If $k < l \leq n,$ then in fact $\rho^{\mathscr{T}}(\phi^{-1}((k,l))) \in \mathbb{Q}(x_t).$ Similarly, if $n \leq k < l,$ then $\rho^{\mathscr{T}}(\phi^{-1}((k,l))) \in \mathbb{Q}(y_{t'}).$
\end{itemize}
\end{theorem}

This map is constructed in Section 2 of \cite{JP}, it is called a Caldero-Chapoton map. One can see that it is an analogue of CC(V) (defined in more categorial way). All these facts mean that the subalgebra $\mathcal{A}_{\mathscr{T}}$ in $\mathbb{Q}(x_t, y_{t'})$ generated by $\rho^{\mathscr{T}}(\mbox{obj} \mathscr{E})$ is actually a disjoint union of two subalgebras in $\mathbb{Q}(x_t)$ and $\mathbb{Q}(y_{t'})$ respectively. 
Consider $\mathfrak{T}_0 = \left\{(n-i-1,n);(n,n+i+1)|i \in \mathbb{N} \right\}, T_0 = \phi^{-1}(\mathfrak{T}_0), \mathscr{T}_0 = \mbox{add}(T_0);$ define $E_0$ like $E$ with corresponding substitution of $T$ by $T_0;$ $\mathscr{E}_0 = \mbox{add}(E_0)$. 
A comparison of $\mathfrak{T}_0$ with an initial quiver of an algebra of type $A_{\infty}$ shows, together with an analogy in the description of flips, that for $\mathfrak{T}_0$ each of these two subalgebras is a cluster algebra of type $A_{\infty}$ (with the same cluster structure induced from the collections of arcs). 

\begin{proposition} \label{foarb}
Consider arbitrary $\mathfrak{T}$ having a fountain at $n.$ There exists such isomorphism
$$ \psi_{T_0, T}: \quad \mathcal{A}_{\mathscr{T_0}} \rightarrow \mathcal{A}_{\mathscr{T}} $$
that 
$$ \psi_{T_0, T} (\rho^{\mathscr{T}_0} (E_0)) = \rho^{\mathscr{T}} (E).$$
\end{proposition}

%For the proof we need one one more definition from \cite{JP}: we will say that the arc $(a, b)$ spans the arc $(c, d)$ if $a \leq c < d < b$ or $a < c < d \leq b.$ Note that an arc does not span itself.

\begin{proof}
By Theorem \ref{obj} $E_0 = E = \phi^{-1}(\left\{(k,l)| (k - n)(l - n) \geq 0 \right\}).$ Define $\psi_{T_0, T}$ at $E_0$ as follows:
$$ \psi_{T_0, T} (\rho^{\mathscr{T}_0} (\phi^{-1}(k,l))) = \rho^{\mathscr{T}} (\phi^{-1}(k,l)).$$
Then by (i) and (ii) in definition of a cluster map extend it to $\rho^{\mathscr{T}}(\mbox{obj} \mathscr{E})$ and to entire $\mathcal{A}_{\mathscr{T_0}}.$ Since (iii) in definition of a cluster map appeals to triangles in $\mathscr{D},$ relations coming from this property are same for $\mathscr{E}_0$ and $\mathscr{E},$ that confirms the correctness of $\psi_{T_0, T}.$ It is clear that $\psi_{T_0, T}$ is an isomorphism.
\end{proof}

Now we will formulate similar results about locally finite collections.

\begin{lemma}
For each natural $n$ quivers $A_n^{alt}$

\begin{center}
\begin{tikzpicture}

  \draw[->,>=stealth] (0,0) node [below] {1} -- (0.95,0);
  \draw[->,>=stealth] (2,0) node [below] {3} -- (1.05,0);
  \draw[->,>=stealth] (2,0) node [below] {3} -- (2.95,0);
  \draw[->,>=stealth] (4,0)  -- (3.05,0);
  \draw (1,0) node [below] {2} -- (1,0);
  \draw (3,0) node [below] {4} -- (3,0);
  \draw (4.2,0) node {.};
  \draw (4.4,0) node {.};
  \draw (4.6,0) node {.};
  \draw (4.8,0) -- (5.75,0) node [below] {\large{n}};
  \fill[black] (0,0) circle (2pt);
  \fill[black] (1,0) circle (2pt);
  \fill[black] (2,0) circle (2pt);
  \fill[black] (3,0) circle (2pt);
  %\fill[black] (4,0) circle (2pt);
  \fill[black] (5.8,0) circle (2pt);
\end{tikzpicture}
\end{center}

and $A_n$

\begin{center}
\begin{tikzpicture}

  \draw[->,>=stealth] (0,0) node [below] {1} -- (0.95,0);
  \draw[->,>=stealth] (1,0) node [below] {2} -- (1.95,0);
  \draw[->,>=stealth] (2,0) node [below] {3} -- (2.95,0);
  \draw[->,>=stealth] (3,0) node [below] {4} -- (4,0);
  \draw (4.2,0) node {.};
  \draw (4.4,0) node {.};
  \draw (4.6,0) node {.};
  \draw[->,>=stealth] (4.8,0) -- (5.75,0) node [below] {\large{n}};
  \fill[black] (1,0) circle (2pt);
  \fill[black] (0,0) circle (2pt);
  \fill[black] (2,0) circle (2pt);
  \fill[black] (3,0) circle (2pt);
  \fill[black] (5.8,0) circle (2pt);
\end{tikzpicture}
\end{center}

are mutation equivalent.
\end{lemma}

\begin{proof}
Prove by induction on $n.$ For $n = 1$ the statement is tautological. Apply to $A_n^{alt}$ one of the following sequences of mutations:

$$\mu_2 \circ \mu_4 \circ \ldots \circ \mu_{n-2} \circ \mu_1 \circ \mu_3 \circ \ldots \circ \mu_{n-1}, \quad  n \vdots 2;$$
$$\mu_1 \circ \mu_3 \circ \ldots \circ \mu_{n-2} \circ \mu_2 \circ \mu_4 \circ \ldots \circ \mu_{n-1}, \quad n \not\vdots 2.$$

We obtain a quiver $A'$ with vertex set $A'_0 = \left\{1,\ldots,n\right\}$ and arrows set $A'_1 = (A_{n-1}^{alt})_1 \cup \left\{(n - 1, n)\right\}.$ By an assumption of induction there is a sequence of mutations leading full subquiver of $A'$ on vertex set $\left\{1,\ldots,n - 1\right\}$ (that is $A_{n-1}^{alt}$) to $A_n.$ Applying this sequence to $A'$ we obtain $A_n,$ that completes the proof.
\end{proof}

For locally finite collections there is an analogue of Theorem \ref{fount}:

\begin{theorem} {\normalfont (\cite{JP})} \label{locfin}
Consider arbitrary locally finite collection $\mathfrak{T}.$ There exists cluster map $\rho^{\mathscr{T}} : \mbox{obj} \mathscr{E} \to \mathbb{Q}(x_t)_{t \in T}$ with following properties:

\begin{itemize}
\item[(i)] $\rho^{\mathscr{T}} (\phi^{-1}(t)) = x_t$ for all $t.$
\item[(ii)] If $\phi^{-1}((k,l)) \in \mathscr{E},$ then $\rho^{\mathscr{T}}(\phi^{-1}((k,l)))$ is a  non-zero Laurent polynomial (in $x_t$ and $y_{t'}$). The coefficients in the numerator of this polynomial are positive integers.
\end{itemize}

\end{theorem}

Consider $\mathfrak{T'_0} = \left\{(-n, n); (-n, n + 1)| n \in \mathbb{N} \right\}$ (such collection is called a {\it leapfrog}. Dual quiver to it is a direct limit of $A_n^{alt},$ and we have the following simple fact. 

\begin{lemma}
A leapfrog provides a filtration
$$A_1^{alt} \hookrightarrow A_2^{alt} \hookrightarrow \ldots \hookrightarrow A_n^{alt} \hookrightarrow \ldots$$
of quivers, where $n-$th quiver is mutationally equivalent to $A_n.$
\end{lemma}

Define as usual $T'_0, \mathscr{T'_0}, E'_0$ and $\mathscr{E'0}.$ We have the following simple consequence.

\begin{corollary}
A subalgebra $\mathcal{A}_{\mathscr{T'_0}}$ in $\mathbb{Q}(x_t)$ generated by $\rho^{\mathscr{T}}(\mbox{obj} \mathscr{E'_0})$ is a cluster algebra of type $A_{\infty}.$
\end{corollary}

\begin{proposition} \label{lfarb}
Consider arbitrary $\mathfrak{T}$ being locally finite. There exists such isomorphism
$$ \psi_{T'_0, T}: \quad \mathcal{A}_{\mathscr{T'_0}} \rightarrow \mathcal{A}_{\mathscr{T}} $$
that 
$$ \psi_{T'_0, T} (\rho^{\mathscr{T'}_0} (E'_0)) = \rho^{\mathscr{T}} (E).$$
\end{proposition}

\begin{proof}
Analogously to the proof of Proposition \ref{foarb}
\end{proof}

\begin{corollary}
An image under Caldero-Chapoton map of each cluster tilting subcategory of $\mathscr{D}$ is either a cluster algebra of type $A_{\infty},$ if corresponding collection of arcs is locally finite, or a disjoint union of two cluster algebras of type $A_{\infty},$ if corresponding collection has a fountain.
\end{corollary}

\begin{remark}
For cluster algebras corresponding to finite quivers Saleh proved in \cite{S} (Corollary 3.4) that each automorphism of a field, which is an isomorphism of two cluster algebras and bijectively maps the set of cluster variables to the set of cluster variables, bijectively maps clusters to clusters.

Propositions \ref{foarb} and \ref{lfarb} provide examples of isomorphisms of cluster algebras (of infinite type), which bijectively map the set of cluster variables to the set of cluster variables, but does not map clusters to clusters. Indeed, there are infinitely many collections $\mathfrak{T}$ of arcs having a fountain at $n,$ which can not be obtained from $\mathfrak{T_0}$ by finitely many flips, e.g. $\mathfrak{T} = \left\{(n, n \pm 2k); (n - 2(k+1), n - 2k); (n + 2k, n + 2(k+1)) |k \in \mathbb{N}  \right\}.$ Clusters of corresponding algebras are images under $\rho^{\mathscr{T}} (\phi^{-1})$ (respectively $\rho^{\mathscr{T_0}}(\phi^{-1})$) of collections which can be obtained from $\mathfrak{T}$ (respectively $\mathfrak{T_0}$) by finitely many flips. That means that $\psi_{T_0, T}$ does not map clusters to clusters. The similar can be said for locally finite collections, as $\mathfrak{T}$ we may take $\left\{(-n, n); (-(n+1), n) | n \in \mathbb{N} \right\}.$ 
\end{remark}

%Thus some of subcategories of $\mathscr{D}$ provide a natural categorification of the algebra which we described. Let us consider other subcategories. Take an arbitrary $\mathfrak{T}$ having a fountain at $n$. Then, thanks to Theorem 6.8 and Proposition 6.3, the properties (i) - (iii) hold, up to substitution of 

We see that, on the one hand, the construction of Holm and J{\o}rgensen includes more information than the our one, since it works with categories. On the other hand, algebras corresponding to cluster tilting subcategories of their category are described in our terms in a very simple manner. Moreover, we defined cluster algebras of types $B_{\infty}, C_{\infty}$ and $D_{\infty}$ for which we proved the Proposition \ref{infan}, while corresponding cluster categories are not yet constructed. A naturalness of a description of these algebras gives the hope that such categories can be defined without big changes in the Holm-J{\o}rgensen construction.

%Explicit formula for $rho^{\mathscr{T}}$ is given in the Section 1 of \cite{JP} and one can check that 

%Jorgensen

% above the line (check)

%Cluster variables

%Cluster relations

\end{document}